\documentclass[12pt,a4paper]{amsart}
\usepackage[hmargin=2.5cm,vmargin=2.5cm]{geometry}
\usepackage{graphicx} % Required for inserting images
\graphicspath{{./Figures}}

\usepackage{amssymb}
\usepackage{amsfonts}
\usepackage{amsmath}
\newtheorem{theorem}{Theorem}

\newtheorem{corollary}{Corollary}

\newtheorem{lemma}{Lemma}

\newtheorem{remark}{Remark}

\usepackage{color}
\begin{document}

\title[Half-space problem for polyatomic gases and entropy inequalities]{The half-space problem of evaporation and condensation for polyatomic gases and entropy inequalities}

\author[N. Bernhoff]{Niclas Bernhoff}
\address{NB: Department of Mathematics and Computer Science, Karlstad University, Universitetsgatan 2, 65188 Karlstad, Sweden}
\email{niclas.bernhoff@kau.se}

\author[S. Brull]{Stéphane Brull}
\address{SB: Institut de Mathématiques de Bordeaux, Bordeaux INP, Université Bordeaux, CNRS, 351, cours de la Libération, Talence, F-33400, France}
\email{stephane.brull@math.u-bordeaux.fr}

\author[E. Wadbro]{Eddie Wadbro}
\address{EW: Department of Mathematics and Computer Science, Karlstad University, Universitetsgatan 2, 65188 Karlstad, Sweden}
\email{eddie.wadbro@kau.se}


\date{\today}

\maketitle

\textbf{Abstract:} 
This study investigates the steady Boltzmann equation in one spatial variable for a polyatomic single-component gas in a half-space. 
Inflow boundary conditions are assumed at the half-space boundary, where particles entering the half-space are distributed as a Maxwellian, an equilibrium distribution characterized by macroscopic parameters of the boundary. 
At the far end, the gas tends to an equilibrium distribution, which is also Maxwellian.
Using conservation laws and an entropy inequality, we derive relations between the macroscopic parameters of the boundary and at infinity required for the existence of solutions.
The relations vary depending on the sign of the Mach number at infinity, which dictates whether evaporation or condensation takes place at the interface between the gas and the condensed phase.
We explore the obtained relations numerically.
This investigation reveals that, although the relations are qualitatively similar for various internal degrees of freedom of the gas, clear quantitative differences exist.

\section{Introduction} 
In their studies, Bobylev et al.~\cite{BGH-01} and Sone et al.~\cite{STG-01} derived qualitative and quantitative estimates of strongly nonequilibrium states using an entropy inequality, bypassing the need to solve the Boltzmann equation.
This work generalizes their results for monatomic single species to include polyatomic single species, exploring possible similarities and differences for various internal degrees of freedom.

Half-space problems are crucial for understanding the asymptotic behavior of the Boltzmann equation at small Knudsen numbers~\cite{Sone-02, Sone-07}.
Here we study the half-space problem for the steady Boltzmann equation, when the distribution function depends on one single spatial variable, while the velocity variable is in the full space.
At infinity, the distribution function tends to an equilibrium distribution---a Maxwellian.
We assume inflow, or complete absorption, boundary conditions at the boundary or at the condensed phase of the gas.
The particles coming from the boundary are distributed as a Maxwellian characterized by the macroscopic parameters of the boundary.
The Maxwellians are characterized by their macroscopic parameters: number density $n$, flow velocity $\mathbf{u}$, and temperature $T$, which can also be expressed in terms of pressure $p$, Mach number $\boldsymbol{\mathcal{M}}$, and temperature $T$.
The relations of the macroscopic parameters at the boundary and at infinity have been well studied for monatomic species, see Refs.~\cite{BGH-01, STG-01, BGS-06, TG-07} and references therein. In particular, we want to mention the seminal papers by Y. Sone, K. Aoki, and coworkers~\cite{SS-90, SAY-86, ASY-90, ANSS-91}, whose results are based on intensive numerical simulations and theoretical considerations.
Assuming the Maxwellian of the boundary to be at rest, the half-space problem can be characterized by the pressure and temperature ratios and the Mach number $\boldsymbol{\mathcal{M}}_{\infty }$ at infinity.
Assuming, for simplicity, that $\boldsymbol{\mathcal{M}}_{\infty}=\left( \mathcal{M}_{\infty 1},0,0\right)$---remaining with three parameters---there will be~\cite{SS-90, SAY-86, ASY-90, ANSS-91}: (i) no solution---three relations between the parameters, and no free parameter---for $\mathcal{M}_{\infty 1}>1$ (supersonic evaporation); (ii) a one-parameter solution---two relations
between the parameters, and one free parameter---for $0\leq \mathcal{M}_{\infty 1}<1$ (subsonic evaporation); (iii) a two-parameter solution---one relation between the parameters, and two free parameters--for $0<\mathcal{M}_{\infty 1}<-1$ (subsonic condensation); (iv) a three-parameter solution---no relation between the parameters, and three free parameters (the domain remains constrained by physical properties)---for $\mathcal{M}_{\infty 1}>-1$ (supersonic condensation).
To the authors' knowledge, the case of polyatomic---including diatomic---gases is not as well studied, despite its importance for a more realistic physical description of many real-world situations.

In this paper, the polyatomicity is modeled by a continuous variable for the (total) internal energy.
To account for the degeneracies of the internal energy levels, a ``weight'' function---density of states---is introduced~\cite{BDLP-94, GP-23, Be-23b}.
The power-law density of the states, which is applied, can be motivated for the case of rotational energy~\cite{BBCG-25}.
For calorically perfect gases, the model enables the recovery of proper forms of the specific internal energy~\cite{AMR-24, BBCG-25} and the ratio of specific heats~\cite{BDLP-94}, and subsequently recovers the compressible Euler equations in the hydrodynamic limit.
The mathematical properties of this model have recently attracted attention~\cite{GP-23, Be-23b, DL-23, BST-24}.
For extensions to multicomponent gases of monatomic and polyatomic species, we refer to the works by Baranger et al.~\cite{BBBD-18}, Alonso et al.~\cite{ACG-24}, and Bernhoff~\cite{Be-24a}, and for chemically reactive gases, we refer to the works by Desvillettes et al.~\cite{DMS-05} and Bernhoff~\cite{Be-24c}. 
For a polyatomic model based on a discrete energy variable, we refer to the research by Ern and Giovangigli~\cite{ErnGio-94}, Groppi and Spiga~\cite{GS-99}, and Bernhoff~\cite{Be-23a, Be-24b}.
Borsoni, Bisi, and Groppi~\cite{BBG-21, BBG-24} introduced a general framework that unifies the different approaches.

Based on the existence results for a general linear half-space problem in kinetic energy by Bernhoff~\cite{Be-23d}, one obtains, for a hard sphere-like model~\cite{Be-23b}, the same number of relations---as in the monatomic case---between the parameters for existence of solutions to the half-space problem for the linearized Boltzmann equation with general inflow boundary conditions.
This suggests that qualitatively similar parameter relationships may be expected for polyatomic gases as for monatomic gases.
In fact, qualitatively similar estimates for the parameters' relations as those for monatomic gases~\cite{BGH-01, STG-01},  are obtained for different internal degrees of freedom.
This is visually illustrated by numerically determined domains, restricted by the obtained estimates.
Here, the quantitative differences for various internal degrees of freedom can also be observed.
We emphasize that all obtained results are independent of any specific choice of collision kernel, relying solely on the conservation laws and the $\mathcal{H}$-theorem---implying the form of the Maxwellian distribution---presented in Section~\ref{MPCO}.

The remainder of the paper is organized as follows.
The kinetic model is presented in Section~\ref{KM}, while the precise formulation of the half-space problem and an existence result for the linearized Boltzmann equation are addressed in Section~\ref{HSP}.
Section~\ref{NCBD} concerns explicit estimates for the macroscopic parameters obtained by using the $\mathcal{H}$-theorem summarized in Lemma~\ref{L2}.
Furthermore, Lemma~\ref{L3} explicitly shows that, similarly as for a monatomic gas, with a Mach number of zero at infinity, there is a unique (Maxwellian) solution.
One key idea in Section~\ref{NCBD} is to represent the entropy flux as the difference of a convex functional for two different values~\cite{BGH-01}. 
A lower estimate for the functional is obtained in Section~\ref{LE}. 
This lower estimate is applied to the current problem in Section~\ref{EPE}.
Finally, Section~\ref{NR} concerns numerical results for the domains, where the obtained estimates for the macroscopic parameters are fulfilled.
The maximal entropy production curves (for a given Mach number at infinity), or surfaces (for a given Mach number at infinity and pressure ratio) for evaporation and condensation, respectively, are also obtained numerically.   

\section{Kinetic model \label{KM}}
In this section, the kinetic model considered is presented.

\subsection{Microscopic model}

In this work, we consider a single species gas of polyatomic molecules with
microscopic mass $m$. Polyatomicity is modeled by a continuous
microscopic internal energy variable $I\in $ $\mathbb{R}_{+}$. The
distribution function of the molecules in the gas is a nonnegative function
of the form $f=f\left( t,\mathbf{x},\boldsymbol{\xi },I\right) $, with time $%
t\in \mathbb{R}_{+}$, microscopic position $\mathbf{x}=\left(
x_{1},x_{2},x_{3}\right) \in \mathbb{R}^{3}$, and microscopic velocity $%
\boldsymbol{\xi }=\left( \xi _{1},\xi _{2},\xi _{3}\right) \in \mathbb{R}%
^{3} $. Denote by $\mathcal{\mathfrak{h:}}=L^{2}\left( d\boldsymbol{\xi \,}%
\mathbf{\,}dI\right) $ the real Hilbert space with inner product%
\begin{equation*}
\left( f,g\right) =\int_{\mathbb{R}^{3}\times \mathbb{R}_{+}}fg\,d%
\boldsymbol{\xi \,}dI\text{ for }f,g\in L^{2}\left( d\boldsymbol{\xi \,}%
\mathbf{\,}dI\right) \text{.}
\end{equation*}

The evolution of the distribution functions is (in the absence of external
forces) described by the Boltzmann equation%
\begin{equation}
\frac{\partial f}{\partial t}+\left( \boldsymbol{\xi }\cdot \nabla _{\mathbf{%
x}}\right) f=Q\left( f,f\right) \text{,}  \label{BE1}
\end{equation}%
where the collision operator $Q=Q\left( f,f\right) $ is a quadratic bilinear
operator that accounts for the change of velocities and internal energies of
particles due to binary collisions (assuming that the gas is rarefied, such
that other collisions are negligible).

A collision (localized in space and time) can be represented by the
microscopic velocities and internal energies of the colliding molecules
before and after the collision, denoted by $\left( \boldsymbol{\xi }%
,I\right) $ and $\left( \boldsymbol{\xi }_{\ast },I_{\ast }\right) $, and $%
\left( \boldsymbol{\xi }^{\prime },I^{\prime }\right) $ and $\left( 
\boldsymbol{\xi }_{\ast }^{\prime },I_{\ast }^{\prime }\right) $,
respectively. Momentum conservation and total energy conservation of the collision
read 
\begin{align*}
\boldsymbol{\xi }+\boldsymbol{\xi }_{\ast } &=\boldsymbol{\xi }^{\prime }+%
\boldsymbol{\xi }_{\ast }^{\prime } \\
\frac{m}{2}\left\vert \boldsymbol{\xi }\right\vert ^{2}+\frac{m}{2}%
\left\vert \boldsymbol{\xi }_{\ast }\right\vert ^{2}+I+I_{\ast } &=\frac{m}{%
2}\left\vert \boldsymbol{\xi }^{\prime }\right\vert ^{2}+\frac{m}{2}%
\left\vert \boldsymbol{\xi }_{\ast }^{\prime }\right\vert ^{2}+I^{\prime
}+I_{\ast }^{\prime }\text{.}
\end{align*}%
In the center-of-mass frame, energy conservation reads%
\begin{equation*}
E:=\frac{m}{4}\left\vert \boldsymbol{\xi }-\boldsymbol{\xi }_{\ast
}\right\vert ^{2}+I+I_{\ast }=\frac{m}{4}\left\vert \boldsymbol{\xi }%
^{\prime }-\boldsymbol{\xi }_{\ast }^{\prime }\right\vert ^{2}+I^{\prime
}+I_{\ast }^{\prime }=:E^{\prime }\text{,}
\end{equation*}%
defining the total energy $E$ in the center-of-mass frame.

Based on the Borgnakke--Larsen model \cite{BL-75} a collision can be
categorized by a parametrization\textbf{\ }$\boldsymbol{\sigma }\in \mathbb{S%
}^{2}$, and $\left( r,R\right) \,\in \left[ 0,1\right] ^{2}$, where $R$
reflects the proportion of the kinetic energy $\dfrac{m}{4}\left\vert 
\boldsymbol{\xi }^{\prime }-\boldsymbol{\xi }_{\ast }^{\prime }\right\vert
^{2}$ relative to the total energy $E$ in the center of mass frame, while $r$
reflects the distribution of the total internal energy $I^{\prime }+I_{\ast
}^{\prime }$ between the two molecules, or, more explicitly%
\begin{equation*}
r=\frac{I^{\prime }}{I^{\prime }+I_{\ast }^{\prime }}\text{, }R=\frac{%
m\left\vert \boldsymbol{\xi }^{\prime }-\boldsymbol{\xi }_{\ast }^{\prime
}\right\vert ^{2}}{4E}\text{ , and }\boldsymbol{\sigma }=\dfrac{\boldsymbol{%
\xi }^{\prime }-\boldsymbol{\xi }_{\ast }^{\prime }}{\left\vert \boldsymbol{%
\xi }^{\prime }-\boldsymbol{\xi }_{\ast }^{\prime }\right\vert }\text{.}
\end{equation*}%
Then the following expressions for microscopic velocities and internal
energies are obtained:%
\begin{eqnarray}
\boldsymbol{\xi }^{\prime } &=&\dfrac{\boldsymbol{\xi }+\boldsymbol{\xi }%
_{\ast }}{2}+\sqrt{\dfrac{RE}{m}}\boldsymbol{\sigma }\text{ and }\boldsymbol{%
\xi }_{\ast }^{\prime }=\dfrac{\boldsymbol{\xi }+\boldsymbol{\xi }_{\ast }}{2%
}-\sqrt{\dfrac{RE}{m}}\boldsymbol{\sigma }\text{, while}  \notag \\
I^{\prime } &=&r(1-R)E\text{ and }I_{\ast }^{\prime }=\left( 1-r\right)
(1-R)E\text{.}  \label{BLP}
\end{eqnarray}

\subsection{Collision operator}

The collision operator in the Boltzmann equation~\eqref{BE1}
for polyatomic molecules can be written in the form 
\begin{equation}
Q(f,f)=\int_{\left( \mathbb{R}^{3}\times \mathbb{R}_{+}\right) ^{3}}W\left( 
\frac{f^{\prime }f_{\ast }^{\prime }}{\varphi ^{\prime }\varphi _{\ast
}^{\prime }}-\frac{ff_{\ast }}{\varphi \varphi _{\ast }}\right) \,d%
\boldsymbol{\xi }_{\ast }d\boldsymbol{\xi }^{\prime }d\boldsymbol{\xi }%
_{\ast }^{\prime }dI_{\ast }dI^{\prime }dI_{\ast }^{\prime }\text{.}
\label{QL1}
\end{equation}%
Here, the density of states $\varphi =\varphi \left( I\right)$---$\varphi
\left( I\right) dI$ representing the number of internal states between $I$
and $I+dI$---is defined to recover a proper form of the specific internal
energy~\cite{AMR-24, BBCG-25}. Typically, a power-law density of states 
\begin{equation}
\varphi \left( I\right) =I^{\delta /2-1}\text{,}  \label{PLDOS}
\end{equation}%
where $\delta >0$ represents the number of internal degrees of freedom of
the molecules, has been considered~\cite{BDLP-94, GP-23, Be-23b}.  From
quantum mechanical results~\cite{Anderson-03, Anderson-06, Atkins-10}, for a
rigid rotor, the power-law density of states $\left( \ref{PLDOS}\right) $
can be motivated for rotational energy~\cite{BBCG-25}, where, e.g. $\delta =2$ for linear molecules and $\delta =3$
for spherical tops (at least approximately). For a calorically perfect gas, the number of internal
degrees of freedom $\delta $ is constant and the power-law density of states~\eqref{PLDOS} may, at least for some purposes, be physically
relevant~\cite{DPT-21, BBCG-25}. For thermally perfect gases, the number of
internal degrees of freedom is depending on the temperature and other
densities of states may have to be considered to be able to capture those
gases. Another approach to capture thermally perfect gases is to introduce a
separate discrete (or, continuous as well) variable for the vibrational part of the internal
energy. For some purposes, it may also be satisfactory to apply the power-law
density of states~\eqref{PLDOS} for thermally perfect gases,
but with an average value for the number of internal degrees of freedom $\delta >0$. 

The main results in this paper are obtained for the power-law
density of states~\eqref{PLDOS}. Here and
below, the standard abbreviations%
\begin{align*}
f_{\ast } &=f\left( t,\mathbf{x}, \boldsymbol{\xi }_{\ast },I_{\ast }\right), &
f^{\prime } &=f\left( t,\mathbf{x},\boldsymbol{\xi }^{\prime},I^{\prime }\right), &
f_{\ast }^{\prime }&=f\left( t,\mathbf{x},\boldsymbol{\xi }_{\ast }^{\prime },I_{\ast }^{\prime }\right)\text{,} \\
\varphi _{\ast } &=\varphi \left( I_{\ast }\right)\text{,} &
\varphi^{\prime }&=\varphi \left( I^{\prime }\right)\text{,} &
\text{and }\varphi _{\ast}^{\prime }&=\varphi \left( I_{\ast }^{\prime }\right) \text{,}
\end{align*}%
are used. The transition probability $W=W(\boldsymbol{\xi },\boldsymbol{\xi }%
_{\ast },I,I_{\ast }\left\vert \boldsymbol{\xi }^{\prime },\boldsymbol{\xi }%
_{\ast }^{\prime },I^{\prime },I_{\ast }^{\prime }\right. )$ is of the form 
\cite{Be-23b} 
\begin{equation}
W=4m\varphi \left( I\right) \varphi \left( I_{\ast }\right)\sigma \frac{\left\vert \boldsymbol{\xi }%
-\boldsymbol{\xi }_{\ast }%
\right\vert }{\left\vert \boldsymbol{\xi }^{\prime }%
-\boldsymbol{\xi }_{\ast }^{\prime }\right\vert }\boldsymbol{\delta 
}_{3}\left( \boldsymbol{\xi }+\boldsymbol{\xi }_{\ast }-\boldsymbol{\xi }%
^{\prime }-\boldsymbol{\xi }_{\ast }^{\prime }\right) \boldsymbol{\delta }%
_{1}\left( E-E^{\prime }\right)\text{,}   \label{TP}
\end{equation}%
where $\boldsymbol{\delta }_{3}$ and $\boldsymbol{\delta }_{1}$ denote
Dirac's delta function in $\mathbb{R}^{3}$ and $\mathbb{R}$, respectively.
The scattering cross section $\sigma =\sigma \left( \left\vert \boldsymbol{%
\xi }-\boldsymbol{\xi }_{\ast }\right\vert ,\left\vert \cos \theta
\right\vert ,I,I_{\ast },I^{\prime },I_{\ast }^{\prime }\right) $, where the
scattering angle $\theta $ is given by $\cos \theta =\dfrac{\boldsymbol{\xi }%
-\boldsymbol{\xi }_{\ast }}{\left\vert \boldsymbol{\xi }-\boldsymbol{\xi }%
_{\ast }\right\vert }\cdot \dfrac{\boldsymbol{\xi }^{\prime }-\boldsymbol{%
\xi }_{\ast }^{\prime }}{\left\vert \boldsymbol{\xi }^{\prime }-\boldsymbol{%
\xi }_{\ast }^{\prime }\right\vert }$, is positive almost everywhere and
satisfies the microreversibility condition
\begin{equation}
\begin{aligned}
&\varphi \left( I\right) \varphi \left( I_{\ast }\right) \left\vert 
\boldsymbol{\xi }-\boldsymbol{\xi }_{\ast }\right\vert ^{2}\sigma \left(
\left\vert \boldsymbol{\xi }-\boldsymbol{\xi }_{\ast }\right\vert
,\left\vert \cos \theta \right\vert ,I,I_{\ast },I^{\prime },I_{\ast
}^{\prime }\right) \\
&\quad =\varphi \left( I^{\prime }\right) \varphi \left( I_{\ast }^{\prime
}\right) \left\vert \boldsymbol{\xi }^{\prime }-\boldsymbol{\xi }_{\ast
}^{\prime }\right\vert ^{2}\sigma \left( \left\vert \boldsymbol{\xi }%
^{\prime }-\boldsymbol{\xi }_{\ast }^{\prime }\right\vert ,\left\vert \cos
\theta \right\vert ,I^{\prime },I_{\ast }^{\prime },I,I_{\ast }\right) \text{%
,} \label{MRC}
\end{aligned}
\end{equation}
as well as the symmetry relations;%
\begin{equation}
\begin{aligned}
\sigma \left( \left\vert \boldsymbol{\xi }-\boldsymbol{\xi }_{\ast
}\right\vert ,\left\vert \cos \theta \right\vert ,I,I_{\ast },I^{\prime
},I_{\ast }^{\prime }\right)  &=\sigma \left( \left\vert \boldsymbol{\xi }-%
\boldsymbol{\xi }_{\ast }\right\vert ,\left\vert \cos \theta \right\vert
,I,I_{\ast },I_{\ast }^{\prime },I^{\prime }\right) \\
&=\sigma \left( \left\vert \boldsymbol{\xi }-\boldsymbol{\xi }_{\ast
}\right\vert ,\left\vert \cos \theta \right\vert ,I_{\ast },I,I_{\ast
}^{\prime },I^{\prime }\right) \text{.} \label{SR}
\end{aligned}
\end{equation}

Applying the Borgnakke--Larsen parametrization~\eqref{BLP}
\cite{BDLP-94, GP-23} the collision operator is recast as 
\begin{align*}
&Q(f,f)=Q_{\delta }(f,f) \\
&\quad =\int_{\mathbb{R}^{3}\times \mathbb{R}_{+}\times \lbrack 0,1]^{2}\mathbb{%
\times S}^{2}}\sigma \left\vert \boldsymbol{\xi }-\boldsymbol{\xi }_{\ast
}\right\vert \left( f^{\prime }f_{\ast }^{\prime }\frac{\varphi \varphi
_{\ast }}{\varphi ^{\prime }\varphi _{\ast }^{\prime }}-ff_{\ast }\right)
\left( 1-R\right) E^{2}d\boldsymbol{\sigma \,}dr\boldsymbol{\,}dR\boldsymbol{%
\,}d\boldsymbol{\xi }_{\ast }\boldsymbol{\,}dI_{\ast }\text{.}
\end{align*}%
In particular, for a power-law density of states~\eqref{PLDOS}
the collision operator becomes~\cite{Be-23b, BDLP-94} 
\begin{equation*}
Q_{\delta }(f,f)=\int_{\mathbb{R}^{3}\times \mathbb{R}_{+}\times \lbrack
0,1]^{2}\mathbb{\times S}^{2}}B_{\delta }\left( \frac{f^{\prime }f_{\ast
}^{\prime }}{\left( I^{\prime }I_{\ast }^{\prime }\right) ^{\delta /2-1}}-%
\frac{ff_{\ast }}{\left( II_{\ast }\right) ^{\delta /2-1}}\right) dA_{\delta
}
\end{equation*}%
for the measure%
\begin{equation*}
dA_{\delta }=\left( r\left( 1-r\right) \right) ^{\delta /2-1}\left(
1-R\right) ^{\delta -1}R^{1/2}\left( II_{\ast }\right) ^{\delta /2-1}d%
\boldsymbol{\sigma \,}dr\boldsymbol{\,}dR\boldsymbol{\,}d\boldsymbol{\xi }%
_{\ast }\boldsymbol{\,}dI_{\ast }
\end{equation*}%
and the collision kernel 
\begin{equation}
B_{\delta }=\frac{\sigma \left\vert \boldsymbol{\xi }-\boldsymbol{\xi }%
_{\ast }\right\vert E^{2}}{\left( 1-R\right) ^{\delta -2}R^{1/2}\left(
r\left( 1-r\right) \right) ^{\delta /2-1}}\text{.}  \notag
\end{equation}%
The collision kernel $B_{\delta }=B_{\delta }(\boldsymbol{\xi },\boldsymbol{%
\xi }_{\ast },I,I_{\ast },r,R,\boldsymbol{\sigma })$ satisfies, by relations~\eqref{MRC} and~\eqref{SR}, the symmetry and
microreversibility relations%
\begin{align*}
B_{\delta }(\boldsymbol{\xi },\boldsymbol{\xi }_{\ast },I,I_{\ast },r,R,%
\boldsymbol{\sigma }) &=B_{\delta }(\boldsymbol{\xi }_{\ast },\boldsymbol{%
\xi },I_{\ast },I,r,R,\boldsymbol{\sigma })\text{,} \\
B_{\delta }(\boldsymbol{\xi },\boldsymbol{\xi }_{\ast },I,I_{\ast },r,R,%
\boldsymbol{\sigma }) &=B_{\delta }(\boldsymbol{\xi },\boldsymbol{\xi }%
_{\ast },I,I_{\ast },1-r,R,-\boldsymbol{\sigma })\text{,} \\
B_{\delta }(\boldsymbol{\xi },\boldsymbol{\xi }_{\ast },I,I_{\ast },r,R,%
\boldsymbol{\sigma }) &=B_{\delta }(\boldsymbol{\xi }^{\prime },\boldsymbol{%
\xi }_{\ast }^{\prime },I^{\prime },I_{\ast }^{\prime },r^{\prime
},R^{\prime },\boldsymbol{\sigma }^{\prime })\text{,}
\end{align*}%
with%
\begin{equation*}
r^{\prime }=\frac{I}{I+I_{\ast }}\text{, }R^{\prime }=\frac{m\left\vert 
\boldsymbol{\xi }-\boldsymbol{\xi }_{\ast }\right\vert ^{2}}{4E}\text{ , and 
}\boldsymbol{\sigma }^{\prime }=\frac{\boldsymbol{\xi }-\boldsymbol{\xi }%
_{\ast }}{\left\vert \boldsymbol{\xi }-\boldsymbol{\xi }_{\ast }\right\vert }%
\text{.}
\end{equation*}%
Below we consider only power-law densities of states $\left( \ref{PLDOS}%
\right) $.

\subsection{Macroscopic quantities}

Macroscopic quantities, i.e., the number density of molecules $n$,
the mass density $\rho $, the flow velocity $\mathbf{u}=\left(
u_{1},u_{2},u_{3}\right) $, the temperature $T$, and the pressure $p$ are
defined by%
\begin{align*}
n &=\left( 1,f\right) \text{, }\rho =mn=\left( m,f\right) \text{, }u_{i}=%
\frac{1}{n}\left( \xi _{i},f\right) \text{, }i=1,2,3\text{, } \\
T &=\frac{2}{\left(3+\delta\right)nk_{B}}\left( \frac{m}{2}\left\vert \boldsymbol{\xi }-\mathbf{u}%
\right\vert ^{2}+I,f\right) \text{, and }p=nk_{B}T=\frac{2}{3+\delta}\left(\frac{m}{2}
\left\vert \boldsymbol{\xi }-\mathbf{u}\right\vert ^{2}+I,f\right) \text{,}
\end{align*}%
where $k_{B}$ denotes the Boltzmann constant. The ratio of specific heats
(also known as the heat capacity ratio or adiabatic index) is 
\begin{equation*}
\gamma =\dfrac{5+\delta }{3+\delta }\text{.}
\end{equation*}%
Especially, for a monatomic gas the ratio of specific heats is $\gamma = 5/3$, while for a diatomic gas $\gamma = 7/5$.

In the hydrodynamic limit, or, when the Knudsen number tends to zero, the
overall evolution of the macroscopic quantities is governed by the compressible Euler equations \cite{Go-05} (in the absence of external forces)%
\begin{align*}
\frac{\partial \rho }{\partial t}+\nabla _{\mathbf{x}}\cdot \left( \rho 
\mathbf{u}\right) &=0\text{,} \\
\frac{\partial \mathbf{u}}{\partial t}+\left( \mathbf{u}\cdot \nabla _{%
\mathbf{x}}\right) \mathbf{u}+\frac{1}{\rho }\nabla _{\mathbf{x}}p &=0\text{,} \\
\frac{\partial }{\partial t}\left( \rho \frac{\left\vert \mathbf{u}%
\right\vert ^{2}}{2}+\frac{1}{\gamma -1}p\right) +\nabla _{\mathbf{x}}\cdot
\left( \rho \frac{\left\vert \mathbf{u}\right\vert ^{2}}{2}+\frac{\gamma }{%
\gamma -1}p\right) \mathbf{u} &=0\text{,}
\end{align*}%
which can be obtained through Chapman--Enskog expansion of the Boltzmann
equation $\left( \ref{BE1}\right) $ for a power-law density of states $%
\left( \ref{PLDOS}\right) $, cf. \cite{BBBD-18}.

The characteristics of the corresponding one-dimensional Euler system are $\left\{
u-c,u,u+c\right\} $, where%
\begin{equation*}
c=\sqrt{\frac{\gamma p}{\rho }}=\sqrt{\frac{\gamma k_{B}T}{m}}
\end{equation*}%
denotes the speed of sound.

\subsection{Main properties of the collision operator\label{MPCO}}

In this section, we recall some main properties of the collision operator~\eqref{QL1}, \eqref{TP}.
In fact, those properties are of main importance later on.

There are five conservation laws \cite{BDLP-94}%
\begin{equation}
\int_{\mathbb{R}^{3}\times \mathbb{R}_{+}}\psi \left( \boldsymbol{\xi }%
\mathbf{,}I\right) Q_{\delta }(f,f)\mathbf{\,}d\boldsymbol{\xi }\mathbf{\,}%
dI=0\text{ for }\psi \left( \boldsymbol{\xi }\mathbf{,}I\right) \in \left\{
1,\xi _{1},\xi _{2},\xi _{3},m\left\vert \boldsymbol{\xi }\right\vert
^{2}+2I\right\}\text{.}  \label{conslaw}
\end{equation}%
The collision operator satisfies the $\mathcal{H}$-theorem \cite{BDLP-94,
De-97, Be-23b}, which states that%
\begin{equation}
\int_{\mathbb{R}^{3}\times \mathbb{R}_{+}}Q_{\delta }(f,f)\log \left(
I^{1-\delta /2}f\right) \mathbf{\,}d\boldsymbol{\xi }\mathbf{\,}dI\leq 0%
\text{,}  \label{HT}
\end{equation}%
with equality in inequality $\left( \ref{HT}\right) $ if and only if 
\begin{equation*}
Q_{\delta }(f,f)=0\text{,}
\end{equation*}%
or, if and only if there exist $n\geq 0$, $\mathbf{u}\in \mathbb{R}^{3}$,
and $T>0$, such that for almost every $\left( \boldsymbol{\xi },I\right) \in 
\mathbb{R}^{3}\times \mathbb{R}_{+}$%
\begin{equation*}
f=M\left( \boldsymbol{\xi },I\right) =\frac{m^{3/2}nI^{\delta /2-1}}{\left(
2\pi \right) ^{3/2}\Gamma \left( \delta /2\right) \left( k_{B}T\right)
^{\left( 3+\delta \right) /2}}\exp \left( -\frac{m(|\boldsymbol{\xi }-%
\mathbf{u}|^{2}+2I)}{2k_{B}T}\right) \text{,}
\end{equation*}%
where $\Gamma $ is the usual Gamma function, is a Maxwellian distribution.

\section{Half-space problem of evaporation and condensation \label{HSP}}

We consider the stationary Boltzmann equation in one spatial dimension.
That is, $f$ depends only on one space variable, henceforth denoted by $x>0$,
but on three velocity variables $\boldsymbol{\xi }=\left( \xi _{1},\xi_{2},\xi _{3}\right) $. Then
\begin{equation}
\xi _{1}\dfrac{\partial f}{\partial x}=Q_{\delta }(f,f),  \label{HSP1}
\end{equation}%
where $f=f(x,\boldsymbol{\xi },I)$ represents the distribution function for
molecules at position $x\in \mathbb{R}_{+}$, with velocity $\boldsymbol{\xi }%
\in \mathbb{R}^{3}$ and internal energy $I\in \mathbb{R}_{+}$, and $\delta
>0 $ denotes the number of internal degrees of freedom.

\subsection{Boundary conditions}

Introduce the notation (where $f=f(\mathbf{\xi })$ possibly can depend on
more variables than $\mathbf{\xi }\in \mathbb{R}^{3}$) 
\begin{equation*}
f_{\pm }(\boldsymbol{\xi })=f_{\pm }(\xi _{1},\xi _{2},\xi _{3})=f(\pm \xi
_{1},\xi _{2},\xi _{3})\hspace*{3mm}\text{for }\xi _{1}>0\text{.}
\end{equation*}%
Assuming complete absorption with a non-drifting incoming Maxwellian distribution $M_{0+}$
%and no net flux at the interface 
and that an
equilibrium distribution $M_{\infty}$ is approached at the far end, we obtain the
boundary conditions%
\begin{equation}
f_{+}(0,\boldsymbol{\xi }\mathbf{,}I)=M_{0+}\text{ and } 
f(x,\boldsymbol{\xi }\mathbf{\mathbf{,}}I)\rightarrow M_{\infty } \text{ as $x\rightarrow \infty$,}  \label{BC}
\end{equation}%
in which
\begin{align*}
M_{0}&=\frac{m^{3/2}n_{0}I^{\delta
/2-1}}{\left( 2\pi \right) ^{3/2}\Gamma \left( \delta /2\right) \left(
k_{B}T_{0}\right) ^{\left( 3+\delta \right) /2}}\exp \left( -\frac{m(|%
\boldsymbol{\xi }|^{2}+2I)}{2k_{B}T_{0}}\right) \text{,} 
\notag \\
M_{\infty } &=\frac{%
m^{3/2}nI^{\delta /2-1}}{\left( 2\pi \right) ^{3/2}\Gamma \left( \delta
/2\right) \left( k_{B}T_{\infty }\right) ^{\left( 3+\delta \right) /2}}\exp
\left( -\frac{m(|\boldsymbol{\xi }-\mathbf{u}_{\infty }|^{2}+2I)}{%
2k_{B}T_{\infty }}\right)\text{,}
\end{align*}%
where we assume that $\mathbf{u}_{\infty }=\left( u,0,0\right) $. The Mach
number at the far end is defined as%
\begin{equation*}
\boldsymbol{\mathcal{M}}_{\infty }=\dfrac{\mathbf{u}_{\infty }}{c}=\left( 
\mathcal{M}_{\infty },0,0\right) \text{, with }\mathcal{M}_{\infty }=\dfrac{u%
}{c}=\sqrt{\dfrac{m}{\gamma k_{B}T_{\infty }}}u\text{.}
\end{equation*}

In Section~\ref{NCBD}, we consider some necessary conditions on the
relations on the boundary data for the possible existence of solutions to the half-space
problem of evaporation and condensation based on entropy inequalities. However, before doing so, we consider a linearized problem in Section~\ref{sec:LHSP}, because we can develop a rigorous theory for the classification of the solution, which is related to that in the original nonlinear problem.

\subsection{Linearized half-space problem\label{sec:LHSP}}
This section concerns a theoretical analysis of a linearized problem, before returning to consider the original nonlinear problem~\eqref{HSP1}, \eqref{BC} in the next section.

After a shift in the velocity variable---$\boldsymbol{\xi }\rightarrow 
\boldsymbol{\xi }+\mathbf{u}_{\infty }$---considering deviations of the far equilibrium
distribution $M=M\left( \boldsymbol{\xi },I\right) =M_{\infty }\left( 
\boldsymbol{\xi }+\mathbf{u}_{\infty },I\right) $ of the form $f=M+\sqrt{M}F$,
discarding quadratic terms---possibly adding an inhomogeneity $%
S=S(x,\boldsymbol{\xi },I)$, such that $( \sqrt{M}S,\psi ) =0$
for $\psi \in \left\{ 1,\xi _{1},\xi _{2},\xi _{3},m\left\vert \boldsymbol{%
\xi }\right\vert ^{2}+2I\right\}$---the system $\left( \ref{HSP1}\right) $
with boundary conditions~\eqref{BC} can be recast as%
\begin{equation}
\begin{cases}
\left( \xi _{1}+u\right) \dfrac{\partial F}{\partial x}+\mathcal{L}%
F=S \\ 
F(0,\boldsymbol{\xi },I)=F_{0}\left( \boldsymbol{\xi },I\right) & \text{ for $\xi _{1}+u>0$} \\ 
F(x,\boldsymbol{\xi }\mathbf{\mathbf{,}}I)\rightarrow 0 & \text{ as $x\rightarrow \infty$,}
\end{cases}%
\label{LHSP}
\end{equation}%
with $F_{0}\left( \boldsymbol{\xi ,}I\right) =M^{-1/2}M_{0}\left( 
\boldsymbol{\xi }+\mathbf{u}_{\infty },I\right) -\sqrt{M}$ and $\mathcal{L}%
F=-2M^{-1/2}Q_{\delta }\left( M,\sqrt{M}F\right) $. For $\delta \geq 2$,
with a scattering cross section of the form
\begin{equation}
\sigma =\sigma _{\delta }=C\dfrac{\left\vert \boldsymbol{\xi }^{\prime }-%
\boldsymbol{\xi }_{\ast }^{\prime }\right\vert }{\left\vert \boldsymbol{\xi }%
-\boldsymbol{\xi }_{\ast }\right\vert }E^{\left( \zeta -1\right) /2-\delta
}\left( I^{\prime }I_{\ast }^{\prime }\right) ^{\delta /2-1}\text{,}
\label{CS}
\end{equation}
or, equivalently, with a collision kernel of the form%
\begin{equation}
B_{\delta }=CE^{\zeta /2}  \label{CK}
\end{equation}%
for some $0\leq \zeta \leq 2$ and some positive constant $C>0$, the
linearized operator $\mathcal{L}$ is an unbounded--for $\zeta >0$%
---self-adjoint nonnegative Fredholm operator \cite{Be-23b} with the domain $%
\mathrm{D}(\mathcal{L})=L^{2}\left( \left( 1+\left\vert \boldsymbol{\xi }%
\right\vert +\sqrt{I}\right) ^{\zeta }d\boldsymbol{\xi \,}\mathbf{\,}%
dI\right) $ \cite{Be-23b,DL-23,Be-24a}, while if $1\leq \zeta \leq 2$ there
is also a positive constant $\widetilde{C}>0$, such that $\left( h,\mathcal{L%
}h\right) \geq \widetilde{C}\left( h,\left( 1+\left\vert \boldsymbol{\xi }%
\right\vert \right) h\right) $ for all $h\in \mathrm{D}(\mathcal{L})\cap 
\mathrm{\rm{Im}}(\mathcal{L})$ \cite{Be-23b}.

An orthonormal basis--also orthogonal with respect to the quadratic form 
$\left( \left. \cdot \right\vert \left( \xi _{1}+u\right) \cdot \right)$--of
the kernel of the linearized operator $\mathcal{L}$ is \cite{Be-23d} 
\begin{equation*}
\begin{cases}
\phi _{1}=\sqrt{\dfrac{M}{2p}}\left( \dfrac{\sqrt{\rho }}{\sqrt{\left(
3+\delta \right) \left( 5+\delta \right) p}}\left( \left\vert \boldsymbol{%
\xi }\right\vert ^{2}+\dfrac{2}{m}I\right) +\xi _{1}\right) \smallskip \\ 
\phi _{2}=\sqrt{\dfrac{M}{2p}}\sqrt{5+\delta }\left( \dfrac{\rho }{\left(
5+\delta \right) p}\left( \left\vert \boldsymbol{\xi }\right\vert ^{2}+%
\dfrac{2}{m}I\right) -1\right) \smallskip \\ 
\phi _{3}=\sqrt{\dfrac{M}{p}}\xi _{2}\smallskip \\ 
\phi _{4}=\sqrt{\dfrac{M}{p}}\xi _{3}\smallskip \\ 
\phi _{5}=\sqrt{\dfrac{M}{2p}}\left( \dfrac{\sqrt{\rho }}{\sqrt{\left(
3+\delta \right) \left( 5+\delta \right) p}}\left( \left\vert \boldsymbol{%
\xi }\right\vert ^{2}+\dfrac{2}{m}I\right) -\xi _{1}\right) \text{.}%
\end{cases}
\end{equation*}%
Denote%
\begin{align*}
\Omega _{+}^{u} &:=\left\{ \phi _{i}\left\vert \left( \left. \phi
_{i}\right\vert \left( \xi _{1}+u\right) \phi _{i}\right) >0\right. \right\}
=\left\{ \phi _{1},...,\phi _{k^{+}}\right\} , \\
\Omega _{0}^{u} &:=\left\{ \phi _{i}\left\vert \left( \left. \phi
_{i}\right\vert \left( \xi _{1}+u\right) \phi _{i}\right) =0\right. \right\}
=\left\{ \phi _{k^{+}+1},...,\phi _{k^{+}+l}\right\}\text{.}
\end{align*}%
Then $\left( k^{+},5-k^{+}-l,l\right) $ is the signature of the restriction
of the quadratic form $\left( \left. \cdot \right\vert \left( \xi
_{1}+u\right) \cdot \right) $ to the kernel of $\mathcal{L}$. Here the
numbers $k^{+}$ and $l$ are the number of positive and zero numbers among $%
\{u-c,u,u,u,u+c\}$, respectively. More precisely,
\begin{equation*}
\Omega _{0}^{-c}=\left\{ \phi _{1}\right\} \text{, }\Omega _{0}^{0}=\left\{
\phi _{2},\phi _{3},\phi _{4}\right\} \text{, and }\Omega _{0}^{c}=\left\{
\phi _{5}\right\} \text{,}
\end{equation*}%
while $\Omega _{0}^{u}=\emptyset $ if $u\notin \left\{ 0,\pm c\right\} $,
and 
\begin{align*}
\Omega _{+}^{u} &=\emptyset \text{ if }u\leq -c\text{; }&
\Omega_{+}^{u}&=\left\{ \phi _{1}\right\} \text{ if }-c<u\leq 0;\text{ } \\
\Omega _{+}^{u} &=\left\{ \phi _{1},\phi _{2},\phi _{3},\phi _{4}\right\} 
\text{ if }0<u\leq c; &
\text{and }\Omega _{+}^{u}&=\left\{ \phi _{1},\phi_{2},\phi _{3},\phi _{4},\phi _{5}\right\} \text{ if }u>c.
\end{align*}%
Consider now the scattering cross section~\eqref{CS} for $%
1\leq \zeta \leq 2$. Moreover, assume that $F_{0+}\in \mathcal{\mathfrak{h}}%
_{+}\cap \mathrm{D}(\mathcal{L})$, with $\mathcal{\mathfrak{h}}_{+}:=\left. 
\mathcal{\mathfrak{h}}\right\vert _{\xi _{1}+u>0}$, and that $e^{\eta x}S(x,%
\boldsymbol{\xi }\mathbf{\mathbf{,}}I)\in L^{2}\left( \mathbb{R}_{+};%
\mathcal{\mathfrak{h}}\right) $ for some positive number $\eta >0$.

\begin{theorem}
\label{T1}\cite{Be-23d} Under the assumptions stated above, imposing $k^{+}+l$ conditions on $F_{0}$, there
exists a unique solution $F=F(x,\boldsymbol{\xi }\mathbf{\mathbf{,%
}}I)$ to the problem $\left( \ref{LHSP}\right) $, such that $e^{\mu x}F(x,%
\boldsymbol{\xi }\mathbf{\mathbf{,}}I)\in L^{2}\left( \mathbb{R}_{+};%
\mathcal{\mathfrak{h}}\right) $ for some $\mu >0$.
\end{theorem}

\begin{corollary}
\label{C1}\cite{Be-23d} Under the assumptions of Theorem \ref{T1} and if the $%
k^{-}=5-k^{+}-l$ parameters $\left( F_{\infty }\mid \phi_{k^{+}+l+1}\right)$,
\ldots, $\left( F_{\infty }\mid \phi_{5}\right) $ are prescribed, then there exists a unique solution $F=F(x,\boldsymbol{%
\xi }\mathbf{\mathbf{,}}I)$ to the problem $\left( \ref{LHSP}\right) $, such
that $e^{\mu x}\left( F(x,\boldsymbol{\xi }\mathbf{\mathbf{,}}I)-F_{\infty
}\right) \in L^{2}\left( \mathbb{R}_{+};\mathcal{\mathfrak{h}}\right) ,\;$%
with $F_{\infty }=\,\underset{x\rightarrow \infty }{\lim }F(x,\boldsymbol{%
\xi }\mathbf{\mathbf{,}}I)\in \ker \mathcal{L}$, for some $\mu >0$.
\end{corollary}

The exponential decay $e^{-\mu x}$ is determined by $\mu =\mu _{u}$ for
fixed $u$. However, the decay is not uniform in $u$ as $u$ tends to some $%
u_{0}\in \left\{ 0,\pm c\right\} $ from the left; it will appear $l$ slowly
varying mode(s) as $u\rightarrow u_{0}^{-}$ (cf. \cite{BG-21} and references
therein). By imposing $l$ extra conditions on the in-data for $u$ less than $%
u_{0}$, the slowly varying modes can be removed in a neighborhood of $u_{0}$%
, i.e., uniform exponential decay can be obtained in a neighborhood of $%
u_{0} $.

Consider now a flow with vanishing flow velocity in the $y$- and $z$-directions,
that is, with $u_{2}=u_{3}=0$. For supersonic evaporation, $u>c$, there are no
free variables, while for subsonic evaporation, $0<u<c$, there is one free
variable. For subsonic condensation, $-c<u<0$, there are two free variables,
and finally, for supersonic condensation, $u<-c$, there are three free variables.

\begin{remark}
Theorem~$\ref{T1}$ and Corollary~$\ref{C1}$ are valid for the more general
scattering cross section%
\begin{equation}
\widetilde{\sigma }=\widetilde{\sigma }_{\delta }=C\left\vert \boldsymbol{%
\xi }^{\prime }-\boldsymbol{\xi }_{\ast }^{\prime }\right\vert ^{\overline{%
\zeta }+1}\left\vert \boldsymbol{\xi }-\boldsymbol{\xi }_{\ast
}\right\vert ^{\overline{\zeta }-1}E^{\left( \zeta-\overline{\zeta } -1\right) /2-\delta
}\left( I^{\prime }I_{\ast }^{\prime }\right) ^{\delta /2-1}\text{,}
\label{CS1}
\end{equation}%
with $1\leq \overline{\zeta }+\zeta \leq 2$, or, equivalently, for the
collision kernel%
\begin{equation*}
\widetilde{B}_{\delta }=C\left\vert \boldsymbol{\xi }^{\prime }-\boldsymbol{%
\xi }_{\ast }^{\prime }\right\vert ^{\overline{\zeta }}\left\vert 
\boldsymbol{\xi }-\boldsymbol{\xi }_{\ast }\right\vert ^{\overline{\zeta }%
}E^{\left(\zeta -\overline{\zeta }\right)/2}\text{, where }1\leq \overline{\zeta }+\zeta \leq 2\text{.}
\end{equation*}
\end{remark}

\begin{remark}
For a scattering cross section~\eqref{CS} with $0\leq \zeta <1$---or more generally, for a scattering cross section~\eqref{CS1} with $0\leq \overline{\zeta }+\zeta <1$---Theorem~$\ref{T1}$ and Corollary~$\ref{C1}$ are valid for $\eta=\mu=0$.
\end{remark}

\begin{remark}
Theorem~$\ref{T1}$ and Corollary~$\ref{C1}$ are valid also for more general
boundary conditions at the interface $x=0$, where the distribution function
of emerging molecules (for which $\xi _{1}+u>0$, in problem~\eqref{LHSP},
for an interface at rest) possibly can depend (partly or
completely) on the distribution function of impinging molecules (for
which $\xi _{1}+u<0$, in problem~\eqref{LHSP}) \cite{Be-23d}.
\end{remark}

The results may be extended to the weakly nonlinear case applying similar methods
as in \cite{BG-21, Go-08}.

Considering---either in a linearized or weakly nonlinear setting---small
deviations $F=F(x,\boldsymbol{\xi },I)$, implies that $F_{0+}(\boldsymbol{\xi }%
,I)=F_{+}(0,\boldsymbol{\xi },I)$ has to be small. While Theorem~\ref{T1}
being valid for any function $F_{0}=F(0,\boldsymbol{\xi },I)$, such that $%
F_{0+}\in \mathcal{\mathfrak{h}}_{+}\cap \mathrm{D}(\mathcal{L})$,
restricting it to be of the form $F_{0}=M^{-1/2}M_{0}\left( \boldsymbol{\xi }%
+\mathbf{u}\right) -\sqrt{M}$, enforces that 
\begin{equation*}
\dfrac{\left\vert T_{0}-T_{\infty }\right\vert }{T_{\infty }}\ll 1\text{ and 
}\left\vert \mathcal{M}_{\infty }\right\vert =\dfrac{\left\vert u\right\vert 
}{c}\ll 1\text{.}
\end{equation*}%
Hence, direct applications of the results for the linearized problem to problem~\eqref{HSP1}, \eqref{BC} are only possible for small Mach numbers $\mathcal{M}_{\infty }$.
On the other hand, in addition to purely mathematical aspects, since $\left\vert F(x, \boldsymbol{\xi },I)\right\vert \ll 1$ for $x\gg 1$, Theorem $\ref{T1}$ imposes $k^{+}+l$ conditions on $F_{L+}=F_{+}(L,\boldsymbol{\xi },I)$---hence, on $f_{L+}=f_{+}(L,\boldsymbol{\xi },I)$ as well---for $L\gg 1$.

\section{Necessary conditions on the boundary data \label{NCBD}}

Now returning to the original nonlinear problem~\eqref{HSP1}, \eqref{BC}, this section concerns some necessary (but not sufficient) restrictions on
the boundary data for the existence of solutions to the half-space problem~\eqref{HSP1}, \eqref{BC}, which are obtained in
a quite similar way as in the papers~\cite{BGH-01, STG-01} for monatomic
species.

Applying the conservation laws~\eqref{conslaw}, we obtain the moments%
\begin{equation*}
\int\limits_{\mathbb{R}^{3}\times \mathbb{R}_{+}}\xi _{1}f(x,\boldsymbol{\xi 
},I\mathbf{)}
\begin{Bmatrix}
1 \\ 
\xi _{1} \\ 
\xi _{2} \\ 
\xi _{3} \\ 
\left\vert \boldsymbol{\xi }\right\vert ^{2}+\dfrac{2}{m}I%
\end{Bmatrix}%
\mathbf{\,}d\boldsymbol{\xi }\mathbf{\,}dI=:
\begin{Bmatrix}
\mathcal{L}_{1} \\ 
\mathcal{L}_{2} \\ 
\mathcal{L}_{3} \\ 
\mathcal{L}_{4} \\ 
\mathcal{L}_{5}%
\end{Bmatrix}
\end{equation*}%
where, noting that $\Gamma \left( \dfrac{\delta }{2}+1\right) =\dfrac{\delta 
}{2}\Gamma \left( \dfrac{\delta }{2}\right) $, 
\begin{align*}
\mathcal{L}_{1} &=n_{\infty }u,\hspace*{3mm}\mathcal{L}_{2}=n_{\infty
}\left( u^{2}+\frac{k_{B}T_{\infty }}{m}\right) ,\hspace*{3mm}\mathcal{L}%
_{3}=\mathcal{L}_{4}=0\text{, and} \\
\mathcal{L}_{5} &=n_{\infty }u\left( u^{2}+\left( 5+\delta \right) \frac{%
k_{B}T_{\infty }}{m}\right) \text{.}
\end{align*}%
Moreover, introducing the $\mathcal{H}$-functional 
\begin{equation}
\Psi \left( f\right) =\int\limits_{\mathbb{R}^{3}\times \mathbb{R}_{+}}\xi
_{1}f\log \left( I^{1-\delta /2}f\right) \mathbf{\,}d\boldsymbol{\xi }%
\mathbf{\,}dI\text{,}  \label{HF}
\end{equation}
and noting that by the $\mathcal{H}$-theorem~\eqref{HT},
\begin{equation}
\frac{d\Psi \left( f\right) }{dx}=\int_{\mathbb{R}^{3}\times \mathbb{R}%
_{+}}Q_{\delta }(f,f)\log \left( I^{1-\delta /2}f\right) \mathbf{\,}d%
\boldsymbol{\xi }\mathbf{\,}dI\leq 0\text{,}  \label{MHT}
\end{equation}%
it follows that
\begin{equation}
\Psi \left( f(x_{1},\boldsymbol{\xi },I)\right) \geq \Psi \left( f(x_{2},%
\boldsymbol{\xi },I)\right) \text{ if }x_{1}\leq x_{2}\text{.}  \label{FIE1}
\end{equation}%
Introduce the generalized $\mathcal{H}$-functional~\cite{BGH-01}%
\begin{equation}
H_{g}\left( f\right) =\int\limits_{\mathbb{R}^{3}\times \mathbb{R}%
_{+}}g\left( \boldsymbol{\xi },I\right) f\log \left( I^{1-\delta /2}f\right) 
\mathbf{\,}d\boldsymbol{\xi }\,dI  \label{GHF}
\end{equation}%
for nonnegative functions $g=g\left( \boldsymbol{\xi },I\right) \in C\left( 
\mathbb{R}^{3}\times \mathbb{R}_{+},\mathbb{R}_{+}\right) $, and moments%
\begin{equation}
\begin{Bmatrix}
\mathcal{\mu }_{1} \\ 
\mathcal{\mu }_{2} \\ 
\mathcal{\mu }_{3} \\ 
\mathcal{\mu }_{4} \\ 
\mathcal{\mu }_{5}%
\end{Bmatrix}%
:=\int\limits_{\mathbb{R}^{3}\times \mathbb{R}_{+}}g\left( 
\boldsymbol{\xi },I\right) f(x,\mathbf{\xi },I)
\begin{Bmatrix}
1 \\ 
\xi _{1} \\ 
\xi _{2} \\ 
\xi _{3} \\ 
\left\vert \boldsymbol{\xi }\right\vert ^{2}+\dfrac{2}{m}I%
\end{Bmatrix}%
\mathbf{\,}d\boldsymbol{\xi }\mathbf{\,}dI\text{,}  \label{Mo}
\end{equation}%
where $\mathcal{\mu }_{1},\mathcal{\mu }_{5}>0$ and $\mathcal{\mu }_{2},%
\mathcal{\mu }_{3},\mathcal{\mu }_{4}\in \mathbb{R}$.

\begin{lemma}
\label{L1}Let $f=f(\boldsymbol{\xi },I)$ be a function such that the
integral $\left( \ref{HF}\right) $ and the moments $\left( \ref{Mo}\right) $
are finite. Assume that there is a Maxwellian%
\begin{equation}
\begin{aligned}
f_{M} &=f_{M}(\boldsymbol{\xi },I)\text{, where} \\
\text{ }f_{M} &= KI^{\delta /2-1}\exp \left( -\left( \beta \left( \left\vert 
\boldsymbol{\xi }\right\vert ^{2}+\frac{2}{m}I\right) +\mathbf{\gamma \cdot
\xi }\right) \right) \text{, }K,\beta \geq 0\text{, }\boldsymbol{\gamma }\in 
\mathbb{R}^{3}\text{,}
\end{aligned}
\label{ConstrMaxw}
\end{equation}
having the same moments~\eqref{Mo} as $f=f(\boldsymbol{\xi },I)$.
Then the Maxwellian $f_{M}$ is uniquely defined and satisfies the
inequality 
\begin{equation*}
H_{g}(f_{M})\leq H_{g}(f).
\end{equation*}
\end{lemma}

The proof is performed in the exact same way as the proof of Lemma $1$ in 
\cite{BGH-01}, but for the sake of completeness we still present it here.

\begin{proof}
By the inequality
\begin{equation}
b\log \frac{b}{a}-\left( b-a\right) \geq 0\text{ for }a,b>0\text{,}
\label{ineq}
\end{equation}%
with equality in inequality $\left( \ref{ineq}\right)$ if and only if $a=b$, and the assumed conservation of the moments  $\left( \ref{Mo}\right)$, we obtain, noting that $\log f_{M} \in \rm{span}\left\{1,\xi_1,\xi_2,\xi_3,|\boldsymbol{\xi}|^2+(2/m)I\right\}$,
\begin{equation*}
H_{g}(f)-H_{g}(f_{M})=\int\limits_{\mathbb{R}^{3}\times \mathbb{R}%
_{+}}g\left( \boldsymbol{\xi },I\right) \left( f\log \frac{f}{f_{M}}-\left(
f-f_{M}\right) \right) \mathbf{\,}d\boldsymbol{\xi }\mathbf{\,}dI\geq 0\text{.}
\end{equation*}

Assume that there is another Maxwellian $\widetilde{f}_{M}$ having the same
moments~\eqref{Mo}. Then $H_{g}(f_{M})=H_{g}(\widetilde{f}_{M})$, and therefore%
\begin{equation*}
\widetilde{f}_{M}\log \frac{\widetilde{f}_{M}}{f_{M}}-\left(\widetilde{f}_{M}-f_{M}\right)=0
\end{equation*}%
on a set of positive measure in $\mathbb{R}^{3}\times \mathbb{R}_{+}$, 
where $g\left( \boldsymbol{\xi },I\right) >0$. Hence, $\widetilde{f}%
_{M}=f_{M}$ for all $\left( \boldsymbol{\xi },I\right) \in \mathbb{R}%
^{3}\times \mathbb{R}_{+}$.
\end{proof}
\begin{remark}
\label{R3}Assuming that $\boldsymbol{\gamma }=\left( 0,\gamma _{2},\gamma
_{3}\right) $ in the Maxwellian $\left( \ref{ConstrMaxw}\right) $, we see that  Lemma $\ref{L1}$ remains valid
even if excluding the second moment $\mathcal{\mu }_{2}$ from the moments $%
\left( \ref{Mo}\right) $ by noting that then $\log
f_{M}\in \mathrm{span}\left\{ 1,\xi _{2},\xi _{3},\left\vert \boldsymbol{\xi 
}\right\vert ^{2}+\dfrac{2}{m}I\right\}$.
\end{remark}
In order to be able to apply Lemma~\ref{L1}, considering $\mathcal{H}$-functional~\eqref{HF}, which does not fulfill the requirements of the lemma itself, we rewrite the $\mathcal{H}$-functional~\eqref{HF} in the following way%
\begin{equation*}
\Psi \left( f\right) =\Psi _{+}\left( f_{+}\right) -\Psi _{+}\left(
f_{-}\right) \text{, }\;\Psi _{+}\left( f_{\pm }\right) =\int\limits_{%
\mathbb{R}_{+}^{3}\times \mathbb{R}_{+}}\xi _{1}f_{\pm }\log \left(
I^{1-\delta /2}f_{\pm }\right) \mathbf{\,}d\boldsymbol{\xi }\mathbf{\,}dI%
\text{,}
\end{equation*}%
where $\mathbb{R}_{+}^{3}=\left\{ \left. \boldsymbol{\xi }\in \mathbb{R}^{3}\right\vert ~\xi _{1}>0\right\}$.
Note that%
\begin{equation*}
\Psi _{+}\left( f_{\pm }\right) =H_{g}\left( f_{\pm }\right) \text{ for }
g(\boldsymbol{\xi },I) = 
\begin{cases}
\xi _{1} & \text{if $\xi _{1}>0$,} \\ 
0 & \text{if $\xi _{1}\leq 0$.}%
\end{cases}
\end{equation*}%
Then 
\begin{equation}
\Psi _{+}\left( M_{0+}\right) -\Psi _{+}\left( f_{-}(0,\boldsymbol{\xi }%
,I)\right) =\Psi \left( f(0,\boldsymbol{\xi },I)\right) \geq \Psi \left(
M_{\infty }\right) .  \label{FIE2}
\end{equation}%
Introduce the "half"-moments $N_{i}^{\pm }$, $i=1,\dotsc ,5$, by 
\begin{equation*}
\int\limits_{\mathbb{R}_{+}^{3}\times \mathbb{R}_{+}}\xi _{1}
\begin{Bmatrix}
1 \\ 
\xi _{1} \\ 
\xi _{2} \\ 
\xi _{3} \\ 
\left\vert \boldsymbol{\xi }\right\vert ^{2}+\dfrac{2}{m}I%
\end{Bmatrix}
f_{\pm }\left( 0,\boldsymbol{\xi },I\right) \,d\boldsymbol{%
\xi }\mathbf{\,}dI=:
\begin{Bmatrix}
N_{1}^{\pm } \\ 
N_{2}^{\pm } \\ 
N_{3}^{\pm } \\ 
N_{4}^{\pm } \\ 
N_{5}^{\pm }%
\end{Bmatrix}
\text{.}
\end{equation*}%
Then, by splitting the velocity variable into its negative and nonnegative parts, it comes that 
\begin{equation}
\begin{aligned}
N_{1}^{+}-N_{1}^{-} &=\mathcal{L}_{1}, \\
N_{2}^{+}+N_{2}^{-} &=\mathcal{L}_{2}, \\
N_{3}^{+}-N_{3}^{-} &=\mathcal{L}_{3}, \\
N_{4}^{+}-N_{4}^{-} &=\mathcal{L}_{4}, \\
N_{5}^{+}-N_{5}^{-} &=\mathcal{L}_{5}\text{.}  
\end{aligned}%
\label{MR}
\end{equation}
On the other hand, by Lemma~\ref{L1}%
\begin{equation*}
  \Psi _{+}\left( f_{-}(0,\boldsymbol{\xi },I)\right) \geq \Psi _{+}\left( f_{M-}\right) \text{.}
\end{equation*}%
Since, at $x=0$%
\begin{equation*}
N_{1}^{+}=n_{0}\sqrt{\frac{k_{B}T_{0}}{2\pi m}}\text{,}\quad 
N_{2}^{+}=\frac{n_{0}k_{B}T_{0}}{2m}\text{,}\quad
N_{5}^{+}=n_{0}\sqrt{\frac{k_{B}T_{0}}{2\pi m}}\left( 4+\delta \right) \frac{k_{B}T_{0}}{m}\text{,}
\end{equation*}%
by the relations $(\ref{MR})$,%
\begin{equation}
\begin{aligned}
N_{1}^{-} &= N_{1}^{+}-\mathcal{L}_{1}=n_{0}\sqrt{\frac{k_{B}T_{0}}{2\pi m}} -n_{\infty }u\geq 0, \\
N_{2}^{-} &= \mathcal{L}_{2}-N_{2}^{+}=n_{\infty }\left( u^{2} + 
\frac{k_{B}T_{\infty }}{m}\right) - \frac{n_{0}k_{B}T_{0}}{2m}\geq 0, \\
N_{5}^{-} &= N_{5}^{+}-\mathcal{L}_{5} \\&=
n_{0}\sqrt{\frac{k_{B}T_{0}}{2\pi m}} \left( \left( 4+\delta \right) \frac{k_{B}T_{0}}{m}\right) - 
n_{\infty}u\left( u^{2}+\left( 5+\delta \right) \frac{k_{B}T_{\infty }}{m}\right) \geq 0\text{.}
\end{aligned}
\label{NMR}
\end{equation}
Introducing the saturation pressures $p_{0}:=n_{0}kT_{0}$ at the interface and $p_{\infty }:=n_{\infty }kT_{\infty }$ at infinity, we obtain, by the
second inequality~\eqref{NMR},
\begin{equation*}
\frac{p_{\infty }}{p_{0}}\geq \frac{1}{2\left( 1+\gamma \mathcal{M}_{\infty
}^{2}\right) },
\end{equation*}%
where $\mathcal{M}_{\infty }=\dfrac{u}{c}=\sqrt{\dfrac{m}{\gamma k_{B}T_{\infty }}}u$ is the Mach number at infinity.
If $u>0$, then, by the first and third inequality $(\ref{NMR})$, respectively, also 
\begin{align*}
  \dfrac{n_{\infty }}{n_{0}} &\leq \frac{1}{u}\sqrt{\frac{k_{B}T_{0}}{2\pi m}} \\
  n_{\infty }u\left( u^{2}+\left( 5+\delta \right) \frac{k_{B}T_{\infty }}{m}\right) 
  &\leq n_{0}\sqrt{\frac{k_{B}T_{0}}{2\pi m}}\left( 4+\delta \right) \frac{k_{B}T_{0}}{m}
\end{align*}
or, equivalently,
\begin{align*}
  \mathcal{M}_{\infty } &\leq \frac{1}{\sqrt{2\pi \gamma }\dfrac{p_{\infty }}{p_{0}}}\sqrt{\frac{T_{\infty }}{T_{0}}}, \\
  \mathcal{M}_{\infty }\left( 3+\delta +\mathcal{M}_{\infty }^{2}\right) 
  &\leq \frac{4+\delta }{\gamma \sqrt{2\pi \gamma }}\frac{1}{\dfrac{p_{\infty }}{p_{0}}\sqrt{\dfrac{T_{\infty }}{T_{0}}}}\text{.}
\end{align*}

Also, by the inequality~\eqref{FIE2}, 
\begin{equation*}
  0\leq \Psi _{+}\left( f_{-}(0,\boldsymbol{\xi },I)\right) \leq \Psi_{+}\left( M_{0+}\right) -\Psi \left( M_{\infty }\right)\text{.}
\end{equation*}%
However,%
\begin{equation*}
  \Psi _{+}\left( M_{0+}\right) =\frac{p_{0}}{\sqrt{2\pi mk_{B}T_{0}}}
  \log \frac{m^{3/2}p_{0}\sqrt{e}}{\left( 2\pi \right) ^{3/2}\Gamma \left( \delta/2\right) \left( ek_{B}T_{0}\right) ^{\left( 5+\delta \right) /2}}\text{,}
\end{equation*}%
while%
\begin{equation*}
  \Psi \left( M_{\infty }\right) =p_{\infty }\frac{u}{k_{B}T_{\infty }}
  \log \frac{m^{3/2}p_{\infty }e}{\left( 2\pi \right) ^{3/2}\Gamma \left( \delta/2\right) \left( ek_{B}T_{\infty }\right) ^{\left( 5+\delta \right) /2}}\text{,}
\end{equation*}%
and hence,%
\begin{equation}
\begin{aligned}
  \Psi _{+}\left( f_{-}(0,\boldsymbol{\xi },I)\right) &\leq 
  \frac{p_{0}}{\sqrt{2\pi mk_{B}T_{0}}}
  \log \frac{m^{3/2}p_{0}\sqrt{e}}
  {\left( 2\pi \right)^{3/2}\Gamma \left( \delta /2\right) \left( ek_{B}T_{0}\right) ^{\left(5+\delta \right) /2}} \\
  &\qquad + 
  p_{\infty }\frac{u}{k_{B}T_{\infty }}
  \log \frac{\left( 2\pi \right)^{3/2}\Gamma \left( \delta /2\right) \left( ek_{B}T_{\infty }\right)^{\left( 5+\delta \right) /2}}
  {m^{3/2}p_{\infty }e}\text{.}
\end{aligned}
\label{PIE1}
\end{equation}

By the general inequality $\left( \ref{ineq}\right) $, 
\begin{align*}
&\Psi _{+}\left( f_{-}(0,\boldsymbol{\xi },I)\right) \\
&\quad \geq \int\limits_{\mathbb{R}_{+}^{3}\times \mathbb{R}_{+}}\xi _{1}f_{-}(0,%
\boldsymbol{\xi },I)\log \left( I^{1-\delta /2}M_{0+}\right) +\xi
_{1}f_{-}(0,\boldsymbol{\xi },I)-\xi _{1}M_{0+}\mathbf{\,}d\boldsymbol{\xi }%
\mathbf{\,}dI \\
&\quad=N_{1}^{-}\left( \log \frac{m^{3/2}n_{0}}{\left( 2\pi \right) ^{3/2}\Gamma
\left( \delta /2\right) \left( k_{B}T_{0}\right) ^{\left( 3+\delta \right)
/2}}+1\right) -\frac{mN_{5}^{-}}{2k_{B}T_{0}}-N_{1}^{+}\text{.}
\end{align*}%
Hence, by the relations $(\ref{MR})$, yields that 
\begin{equation}
\begin{aligned}
&\Psi _{+}\left( f_{-}(0,\boldsymbol{\xi },I)\right) \\
&\geq \left( N_{1}^{+}-\mathcal{L}_{1}\right) \log \frac{m^{3/2}p_{0}}{%
\left( 2\pi \right) ^{3/2}\Gamma \left( \delta /2\right) \left(
k_{B}T_{0}\right) ^{\left( 5+\delta \right) /2}}-\mathcal{L}_{1}+\frac{%
\mathcal{L}_{5}-N_{5}^{+}}{2k_{B}T_{0}}m \\
&=\left( \frac{p_{0}}{\sqrt{2\pi mk_{B}T_{0}}}-p_{\infty }\frac{u}{%
k_{B}T_{\infty }}\right) \left( \log \frac{m^{3/2}p_{0}}{\left( 2\pi \right)
^{3/2}\Gamma \left( \delta /2\right) \left( k_{B}T_{0}\right) ^{\left(
5+\delta \right) /2}}-\frac{4+\delta }{2}\right) \\
&-\frac{6+\delta }{2}p_{\infty }\frac{u}{k_{B}T_{\infty }}+\frac{5+\delta }{%
2}p_{\infty }\frac{u}{k_{B}T_{0}}\left( 1+\frac{\mathcal{M}_{\infty }^{2}}{%
3+\delta }\right) \text{.}
\end{aligned}
\label{PIE2}
\end{equation}

Now, by combining the inequalities~\eqref{PIE1} and~\eqref{PIE2}, we obtain the inequality%
\begin{equation}
u\left( \log \frac{p_{0}}{p_{\infty }}+\frac{5+\delta }{2}\log \frac{%
T_{\infty }}{T_{0}}+\frac{5+\delta }{2}\left( 1-\left( 1+\frac{\mathcal{M}%
_{\infty }^{2}}{3+\delta }\right) \frac{T_{\infty }}{T_{0}}\right) \right)
\geq 0\text{,}  \label{PIE3}
\end{equation}%
which, for $u>0$, reduces to 
\begin{equation*}
\log \frac{p_{0}}{p_{\infty }}+\frac{5+\delta }{2}\log \frac{T_{\infty }}{%
T_{0}}+\frac{5+\delta }{2}\left( 1-\left( 1+\frac{\mathcal{M}_{\infty }^{2}}{%
3+\delta }\right) \frac{T_{\infty }}{T_{0}}\right) \geq 0\text{,}
\end{equation*}%
or, equivalently,%
\begin{equation*}
\frac{p_{\infty }}{p_{0}}\leq \left( \frac{T_{\infty }}{T_{0}}\right) ^{%
\tfrac{5+\delta }{2}}\exp \left( \frac{5+\delta }{2}\left( 1-\left( 1+\frac{%
\mathcal{M}_{\infty }^{2}}{3+\delta }\right) \frac{T_{\infty }}{T_{0}}%
\right) \right)\text{.}
\end{equation*}%
Since the function $h(x)=x^{\tfrac{5+\delta }{2}}\exp \left( \frac{5+\delta}{2}(1-x)\right)$, $x>0$, obtains a maximum at $x=1$, we obtain the inequality
\begin{equation*}
\begin{aligned}
\frac{p_{\infty }}{p_{0}} &\leq \frac{\left( \dfrac{T_{\infty }}{T_{0}}
\left( 1+\dfrac{\mathcal{M}_{\infty }^{2}}{3+\delta }\right) \right)
^{\left( 5+\delta \right) /2}\exp \left( \dfrac{5+\delta }{2}\left( 1-\left(
1+\dfrac{\mathcal{M}_{\infty }^{2}}{3+\delta }\right) \dfrac{T_{\infty }}{%
T_{0}}\right) \right) }{\left( 1+\dfrac{\mathcal{M}_{\infty }^{2}}{3+\delta }%
\right) ^{\left( 5+\delta \right) /2}} \\
&\leq \frac{1}{\left( 1+\dfrac{\mathcal{M}_{\infty }^{2}}{3+\delta }\right)^{\left( 5+\delta \right) /2}}\text{.}
\end{aligned}
\end{equation*}

For $u<0$, inequality $\left( \ref{PIE3}\right) $ reduces to 
\begin{equation*}
\log \frac{p_{0}}{p_{\infty}}+\frac{5+\delta }{2}\log \frac{T_{\infty }}{%
T_{0}}+\frac{5+\delta }{2}\left( 1-\left( 1+\frac{\mathcal{M}_{\infty}^{2}%
}{3+\delta }\right) \frac{T_{\infty }}{T_{0}}\right) \leq 0
\end{equation*}%
or, equivalently,%
\begin{equation*}
\frac{p_{\infty }}{p_{0}}\geq \left( \frac{T_{\infty }}{T_{0}}\right)
^{\left( 5+\delta \right) /2}\exp \left( \frac{5+\delta }{2}\left( 1-\left(
1+\frac{\mathcal{M}_{\infty }^{2}}{3+\delta }\right) \frac{T_{\infty }}{T_{0}%
}\right) \right)\text{.}
\end{equation*}

The following lemma summarizes the results obtained in this section on the necessary conditions on the boundary data for the existence of solutions to the half-space problem.

\begin{lemma}
\label{L2}For the half-space problem~\eqref{HSP1}, with boundary conditions~\eqref{BC}, to possibly have any solutions, the following relations between the parameters of the two Maxwellians at the condensed interface and the uniform phase at infinity are forced to be fulfilled.

i) (Overall condition.) For all $u$ (or, equivalently, all $\mathcal{M}_{\infty }$), we have the necessary condition
\begin{equation}
\frac{p_{\infty }}{p_{0}}\geq \frac{1}{2\left( 1+\gamma \mathcal{M}_{\infty
}^{2}\right) }\text{.}  \label{NC1}
\end{equation}

ii) (Evaporation condition.) For all $u>0$ (or, equivalently, all $\mathcal{M}_{\infty }>0$), we have the additional necessary conditions
\begin{equation}
\begin{aligned}
\mathcal{M}_{\infty } &\leq \frac{1}{\sqrt{2\pi \gamma }\dfrac{p_{\infty }}{%
p_{0}}}\sqrt{\frac{T_{\infty }}{T_{0}}} \\
\mathcal{M}_{\infty }\left( 3+\delta +\mathcal{M}_{\infty }^{2}\right) 
&\leq \frac{4+\delta }{\gamma \sqrt{2\pi \gamma }}\frac{1}{\dfrac{p_{\infty
}}{p_{0}}\sqrt{\dfrac{T_{\infty }}{T_{0}}}} \\
\frac{p_{\infty }}{p_{0}} &\leq \frac{1}{\left( 1+\dfrac{\mathcal{M}%
_{\infty }^{2}}{3+\delta }\right) ^{\left( 5+\delta \right) /2}}\text{.}
\end{aligned}
\label{NCP1}
\end{equation}

iii) (Condensation condition.) For all $u<0$ (or, equivalently, all $\mathcal{M}_{\infty }<0$), we have the additional necessary condition
\begin{equation}
\frac{p_{\infty }}{p_{0}}\geq \left( \frac{T_{\infty }}{T_{0}}\right)
^{\left( 5+\delta \right) /2}\exp \left( \frac{5+\delta }{2}\left( 1-\left(
1+\frac{\mathcal{M}_{\infty }^{2}}{3+\delta }\right) \frac{T_{\infty }}{T_{0}%
}\right) \right) \text{.}  \label{NCN1}
\end{equation}
\end{lemma}

In addition to the above, we end this section by proving the following additional necessary condition.
\begin{lemma}
\label{L3}
For the half-space problem~\eqref{HSP1}, with boundary conditions~\eqref{BC}, to possibly have any solutions for the case $u=0$ (or, equivalently, $\mathcal{M}_{\infty }=0$), we have the necessary condition%
\begin{equation}
\frac{p_{\infty }}{p_{0}}=\frac{T_{\infty }}{T_{0}}=1\text{.}  \label{TC}
\end{equation}%
Moreover, under the condition~\eqref{TC} there exists a unique solution $f=f(x,\boldsymbol{\xi },I)=M_{0}(\boldsymbol{\xi },I)$.
\end{lemma}

\begin{proof}
The proof is similar to the one for monatomic species in \cite{BGS-06}.

For $u=0$, $N_{i}^{+}-N_{i}^{-}=\mathcal{L}_{i}=0$ for $i\in \left\{1,3,4,5\right\} $.
Then, by Remark~\ref{R3} and inequality~\eqref{FIE2}, it holds
\begin{equation*}
0=\Psi \left( M_{\infty }\right) \leq \Psi \left( f(x,\boldsymbol{\xi }%
,I)\right) \leq \Psi \left( f(0,\boldsymbol{\xi },I)\right) =\Psi _{+}\left(
M_{0+}\right) -\Psi _{+}\left( f_{-}(0,\boldsymbol{\xi },I)\right) \leq 0%
\text{.}
\end{equation*}
Consequently,
\begin{equation*}
\int_{\mathbb{R}^{3}\times \mathbb{R}_{+}}Q_{\delta }(f,f)\log \left(
I^{1-\delta /2}f\right) \mathbf{\,}d\boldsymbol{\xi }\mathbf{\,}dI=\frac{%
d\Psi \left( f\right) }{dx}=0\text{,}
\end{equation*}
and, hence,
\begin{equation*}
Q_{\delta }(f,f)=0\text{.}
\end{equation*}
Then $f=f(x,\boldsymbol{\xi },I)$ is a local Maxwellian for any $x\in \mathbb{R}_{+}$, and by the Boltzmann equation~\eqref{HSP1}, $f=f(x,\boldsymbol{\xi },I)$ is constant in $x$ for any $\boldsymbol{\xi } \in \mathbb{R}^{3}$ such that $\xi _{1}\neq 0$.
Hence, $f=f(x,\boldsymbol{\xi },I)=M_{0}(\boldsymbol{\xi },I)=M_{\infty }(\boldsymbol{\xi },I)$. 
\end{proof}

Furthermore, by inequality~\eqref{PIE1}
\begin{multline}
\Psi _{+}\left( f_{-}(0,\boldsymbol{\xi },I)\right) \leq \frac{p_{0}}{\sqrt{%
2\pi mk_{B}T_{0}}}\log \frac{m^{3/2}p_{0}\sqrt{e}}{\left( 2\pi \right)
^{3/2}\Gamma \left( \delta /2\right) \left( ek_{B}T_{0}\right) ^{\left(
5+\delta \right) /2}}\\
+p_{\infty }\sqrt{\frac{\gamma }{mk_{B}T_{\infty }}}\mathcal{M}_{\infty
}\log \frac{\left( 2\pi \right) ^{3/2}\Gamma \left( \delta /2\right) \left(
ek_{B}T_{\infty }\right)^{\left( 5+\delta \right) /2}}{m^{3/2}p_{\infty }e}\text{.}
\label{PIE4}
\end{multline}%
This latest inequality will be considered more closely in Section \ref{EPE}, but first, motivated by~\cite{BGH-01}, we study lower estimates for $\Psi_{+}\left(f\right)$ in the next section.

\section{Lower estimate for $\Psi_{+}\left(f\right)$ \label{LE}}
For the sake of simplicity, below we consider only symmetric flows such that
\begin{equation*}
f=f\left( x,\boldsymbol{\xi },I\right) =f\left( x,\xi _{1},r,I\right) \text{
for }r=\sqrt{\xi _{2}^{2}+\xi _{3}^{2}}\text{.}
\end{equation*}

Denote%
\begin{equation}
2\pi \int\limits_{0}^{\infty }\int\limits_{0}^{\infty
}\int\limits_{0}^{\infty }zr
\begin{Bmatrix}
1 \\ 
z \\ 
z^{2}+r^{2}+\dfrac{2}{m}I%
\end{Bmatrix}
f(z,r,I)\mathbf{\,}dzdrdI=
\begin{Bmatrix}
N_{1} \\ 
N_{2} \\ 
N_{5}%
\end{Bmatrix}
\text{.}  \label{AMD}
\end{equation}%
We will consider the following problem motivated by~\cite{BGH-01}.
The solution also follows the lines of the corresponding solution in~\cite{BGH-01}.

\textbf{Problem:} Find%
\begin{equation*}
\begin{aligned}
F(N_{1},N_{2},N_{5}) &= \min \Psi _{+}\left( f\right), \\
\Psi _{+}\left( f\right) &= 2\pi \int\limits_{0}^{\infty} \int\limits_{0}^{\infty }\int\limits_{0}^{\infty}zrf\log \left(I^{1-\delta /2}f\right) \mathbf{\,}dzdrdI\text{.}
\end{aligned}
\end{equation*}
\textbf{Solution: }We first construct the Maxwellian
\begin{equation}
f_{M}=\left( \frac{2}{m}\right) ^{\delta /2}\frac{a\beta ^{3+\delta }}{%
\Gamma \left( \delta /2\right) \pi }I^{\delta /2-1}e^{-\beta ^{2}\left(
\left\vert z-w\right\vert ^{2}+r^{2}+2I/m\right) }\text{,}  \label{AMa}
\end{equation}
with the moments $N_{1}$, $N_{2}$, and $N_{5}$.
Substituting the Maxwellian~\eqref{AMa} into equations~\eqref{AMD}, we obtain the following equations
\begin{equation}
N_{1}=\frac{a}{\beta }I_{1}(\beta w),\hspace*{3mm}N_{2}=\frac{a}{\beta ^{2}}%
I_{2}(\beta w),\hspace*{3mm}N_{5}=\frac{a}{\beta ^{3}}\left( I_{3}(\beta
w)+I_{1}(\beta w)\left( 1+\frac{\delta }{2}\right) \right) \text{,}
\label{AME}
\end{equation}
where
\begin{equation}
I_{n}(s)=\int\limits_{0}^{\infty }z^{n}e^{-(z-s)^{2}}\mathbf{\,}%
dz=\int\limits_{-s}^{\infty }\left( z+s\right) ^{n}e^{-z^{2}}\mathbf{\,}%
dz\geq 0\text{.}  \label{AID}
\end{equation}
Then
\begin{align*}
N_{2} &=\frac{N_{1}}{\beta }\frac{I_{2}(s)}{I_{1}(s)}, &
N_{5}&=\frac{N_{1}}{\beta ^{2}}\left( \frac{I_{3}(s)}{I_{1}(s)}+1+\frac{\delta }{2}\right)\text{,} \\
\frac{N_{1}N_{5}}{N_{2}^{2}} &=\frac{I_{1}(s)\left[ I_{3}(s)+I_{1}(s)\left(
1+\delta /2\right) \right] }{I_{2}^{2}(s)}, & s&=\beta w\text{,}
\end{align*}%
and%
\begin{equation*}
a=\frac{N_{1}^{2}}{N_{2}}\dfrac{I_{2}(s)}{I_{1}^{2}(s)}, \quad
\beta =\frac{N_{1}}{N_{2}}\dfrac{I_{2}(s)}{I_{1}(s)}, \quad
w=\frac{N_{2}}{%
N_{1}}\dfrac{sI_{1}(s)}{I_{2}(s)}\text{.}
\end{equation*}

We have the recursion formula%
\begin{equation*}
\begin{aligned}
I_{n}(s) &=\int\limits_{0}^{\infty }z^{n}e^{-(z-s)^{2}}\mathbf{\,}%
dz=\int\limits_{0}^{\infty }z^{n-1}\left( z-s\right) e^{-(z-s)^{2}}\mathbf{\,%
}dz+sI_{n-1}(s) \\
&=\int\limits_{0}^{\infty }\frac{n-1}{2}z^{n-2}e^{-(z-s)^{2}}\mathbf{\,} dz+sI_{n-1}(s)
=sI_{n-1}(s)+\frac{n-1}{2}I_{n-2}(s)\text{.}
\end{aligned}
\end{equation*}%
Hence, we obtain
\begin{equation}
\begin{aligned}
I_{0}(s) &=\int\limits_{0}^{\infty }e^{-(z-s)^{2}}\mathbf{\,}%
dz=\int\limits_{-s}^{\infty }e^{-z^{2}}\mathbf{\,}dz, \\
I_{1}(s) &=\int\limits_{0}^{\infty }ze^{-(z-s)^{2}}\mathbf{\,}dz 
\\ &= \int\limits_{0}^{\infty }\left( z-s\right) e^{-(z-s)^{2}}\mathbf{\,}%
dz+sI_{0}(s)=sI_{0}(s)+\frac{1}{2}e^{-s^{2}}, \\
I_{2}(s) &=sI_{1}(s)+\frac{1}{2}I_{0}(s)=\left( s^{2}+\frac{1}{2}\right)
I_{0}(s)+\frac{s}{2}e^{-s^{2}}, \\
I_{3}(s) &=sI_{2}(s)+I_{1}(s)=\left( s^{3}+\frac{3}{2}s\right) I_{0}(s)+%
\frac{s^{2}+1}{2}e^{-s^{2}}, \\
I_{4}(s) &=sI_{3}(s)+\frac{3}{2}I_{2}(s)=\left( s^{4}+3s^{2}+\frac{3}{4}%
\right) I_{0}(s)+\left( \frac{s^{3}}{2}+\frac{5}{4}s\right) e^{-s^{2}}, \\
I_{5}(s) &=sI_{4}(s)+2I_{3}(s) \\&=\left( s^{5}+5s^{3}+\frac{15}{4}s\right)
I_{0}(s)+\left( \frac{s^{4}}{2}+\frac{9}{4}s^{2}+1\right) e^{-s^{2}}.%\text{etc.}  
\end{aligned}
\label{AIR}
\end{equation}%
Using inequality~\eqref{AID}, we obtain for $s<0$ 
\begin{equation*}
\frac{1+\frac{5}{2s^{2}}}{1+\frac{3}{s^{2}}+\frac{3}{4s^{4}}}\frac{e^{-s^{2}}%
}{2\left\vert s\right\vert } \leq I_{0}(s)\leq \frac{1+\frac{9}{2s^{2}}+%
\frac{2}{s^{4}}}{1+\frac{5}{s^{2}}+\frac{15}{4s^{4}}}\frac{e^{-s^{2}}}{%
2\left\vert s\right\vert}.
\end{equation*}
It might be interesting to note that:
\begin{equation*}
\frac{1+\frac{5}{2s^{2}}}{1+\frac{3}{s^{2}}+\frac{3}{4s^{4}}} \approx 
\frac{1+\frac{9}{2s^{2}}+\frac{2}{s^{4}}}{1+\frac{5}{s^{2}}+\frac{15}{4s^{4}}%
}\approx 1-\frac{1}{2s^{2}}+\frac{3}{4s^{4}}-\frac{15}{8s^{6}}\quad 
\text{as $s\rightarrow -\infty.$}
\end{equation*}%
It follows that%
\begin{align*}
I_{0}(s) &\approx \left( -\frac{1}{2s}+\frac{1}{4s^{3}}-\frac{3}{8s^{5}}+\frac{15}{16s^{7}}\right) e^{-s^{2}} \\
I_{1}(s) &\approx \frac{1}{4s^{2}}e^{-s^{2}} \\
I_{2}(s) &\approx \frac{-1}{4s^{3}}e^{-s^{2}} \\
I_{3}(s) &\approx \frac{3}{8s^{4}}e^{-s^{2}}
\end{align*}%
as $s\rightarrow -\infty $, and
\begin{align*}
I_{0}(s) &\rightarrow \sqrt{\pi }\text{ } \\
I_{1}(s) &\approx \sqrt{\pi }s \\
I_{2}(s) &\approx \sqrt{\pi }s^{2}\left( 1+\frac{1}{2s^{2}}\right) \\
I_{3}(s) &\approx \sqrt{\pi }s^{3}\left( 1+\frac{3}{2s^{2}}\right)
\end{align*}%
as $s\rightarrow \infty $.

We define 
\begin{equation*}
\Phi (s):=\frac{I_{1}(s)\left[ I_{3}(s)+I_{1}(s)\left( 1+\delta /2\right)\right]}{I_{2}^{2}(s)}\text{.}
\end{equation*}
Notice that if $N_1N_5/N_2>1$, then the parameter $s\in (-\infty ,\infty )$ can be obtained from the equation 
\begin{equation}
\Phi (s)=\frac{N_{1}N_{5}}{N_{2}^{2}}\text{.}
\label{AE}
\end{equation}
Note that by relations~\eqref{AIR}, we also obtain that
\begin{equation*}
\Phi (s)=\frac{4s^{2}A^{2}+2sA+2\left( 4+\delta \right) A^{2}}{\left(
1+2sA\right) ^{2}}\text{,}
\end{equation*}
where
\begin{equation*}
\hspace*{3mm}A:=\frac{I_{1}(s)}{I_{0}(s)}=s+\frac{1%
}{2e^{s^{2}}\int\limits_{-s}^{\infty }e^{-z^{2}}\mathbf{\,}dz}\text{.}
\end{equation*}%
In fact, 
\begin{equation*}
\Phi (s)\approx \frac{1+\frac{5+\delta }{2s^{2}}}{\left( 1+\frac{1}{2s^{2}}%
\right) ^{2}}\approx 1+\frac{3+\delta }{2s^{2}}\rightarrow 1\quad \text{as $s\rightarrow \infty.$}
\end{equation*}%
and
\begin{equation*}
\Phi(s)\approx s^{2}\left( 1+\frac{\delta }{2}\right) \rightarrow \infty 
\quad \text{as $s\rightarrow -\infty$.}
\end{equation*}%
Therefore, $\Phi(s)$ takes all values in $\left( 1,\infty \right)$.
Hence, equation~\eqref{AE} has a solution $s=\widetilde{s}
$ for all $\dfrac{N_{1}N_{5}}{N_{2}^{2}}>1$. Moreover, the root is unique by
Lemma $\ref{L1}$, since $\widetilde{s}=\beta w$ for uniquely defined numbers 
$\beta$ and $w$. Although we cannot express $\widetilde{s}$ explicitly
through $\dfrac{N_{1}N_{5}}{N_{2}^{2}}$, we can obtain some asymptotic
expressions. In more detail, 
\begin{equation*}
\frac{N_{1}N_{5}}{N_{2}^{2}}\approx 1+\frac{3+\delta }{2s^{2}}\Rightarrow 
\widetilde{s}\approx \sqrt{\frac{3+\delta }{2}}\left( \frac{N_{1}N_{5}}{%
N_{2}^{2}}-1\right) ^{-1/2}\hspace*{3mm}\text{as }s\rightarrow \infty
\end{equation*}%
and%
\begin{equation*}
\frac{N_{1}N_{5}}{N_{2}^{2}}\approx s^{2}\left( 1+\frac{\delta }{2}\right)
\Rightarrow \widetilde{s}\approx -\sqrt{\frac{2N_{1}N_{5}}{\left( 2+\delta
\right) N_{2}^{2}}}\hspace*{3mm}\text{as }s\rightarrow -\infty\text{.}
\end{equation*}%
It is also clear that $\widetilde{s}$ depends on $\dfrac{N_{1}N_{5}}{%
N_{2}^{2}}$ monotonically. Then the first part is finalized and we can move
on to the construction of the function%
\begin{equation*}
F(N_{1},N_{2},N_{5})=\Psi _{+}\left( f_{M}\right) =2\pi
\int\limits_{0}^{\infty }\int\limits_{0}^{\infty }\int\limits_{0}^{\infty
}zrf_{M}\log \left( I^{1-\delta /2}f_{M}\right) \mathbf{\,}dzdrdI\text{.}
\end{equation*}

Noting that%
\begin{equation*}
\log \left( I^{1-\delta /2}f_{M}\right) =\log \frac{2^{\delta /2}a\beta
^{3+\delta }}{m^{\delta /2}\Gamma \left( \delta /2\right) \pi }-\beta
^{2}\left( w^{2}-2zw+z^{2}+r^{2}+\frac{2}{m}I\right) \text{,}\hspace*{3mm}%
s=\beta w\text{,}
\end{equation*}%
then by expressions~\eqref{AME} for the moments~\eqref
{AMD}:
\begin{equation*}
\begin{aligned}
F(N_{1},N_{2},N_{5})
&=
N_{1}\left( \log \frac{2^{\delta /2}a\beta^{3+\delta }}{m^{\delta /2}\Gamma \left( \delta /2\right) \pi }-s^{2}\right)
+ 2s\beta N_{2}-\beta ^{2}N_{5}
\\&=
N_{1}\log \frac{2^{\delta /2}a\beta ^{3+\delta }}{m^{\delta /2}\Gamma
\left( \delta /2\right) \pi }-\frac{a}{\beta }\int\limits_{-s}^{\infty
}(z+s)\left( 1+\frac{\delta }{2}+z^{2}\right) e^{-z^{2}}\mathbf{\,}dz
\\ &=
N_{1}\left( \log \frac{2^{\delta /2}a\beta ^{3+\delta }}{m^{\delta
/2}\Gamma \left( \delta /2\right) \pi }-1-\frac{\delta }{2}-\theta
(s)\right) \text{,}
\end{aligned}
\end{equation*}%
where 
\begin{equation*}
\theta (s)=\frac{\int\limits_{-s}^{\infty }(z+s)z^{2}e^{-z^{2}}\mathbf{\,}dz%
}{\int\limits_{-s}^{\infty }(z+s)e^{-z^{2}}\mathbf{\,}dz}=\frac{%
I_{3}(s)-2sI_{2}(s)+s^{2}I_{1}(s)}{I_{1}(s)}=\frac{sI_{0}(s)+e^{-s^{2}}}{%
2sI_{0}(s)+e^{-s^{2}}}\text{,}
\end{equation*}%
or, equivalently, by the relations $\left( \ref{AIR}\right) $,%
\begin{equation}
\theta (s)=\frac{1}{2}+\frac{e^{-s^{2}}}{4I_{1}(s)}=1-\frac{sI_{0}(s)}{%
2I_{1}(s)}=1+s^{2}-s\frac{I_{2}(s)}{I_{1}(s)},  \label{AT}
\end{equation}%
has typical values%
\begin{align*}
\theta (s) &\rightarrow \frac{1}{2} \quad\text{as $s\rightarrow \infty$,} \\
\theta (0) &=1, \\
\theta (s) &\approx s^{2}\rightarrow \infty \quad
\text{as $s\rightarrow -\infty.$}
\end{align*}
Note that
\begin{equation*}
a\beta ^{3+\delta }=\frac{N_{1}^{5+\delta }}{N_{2}^{4+\delta }}\dfrac{%
I_{2}^{4+\delta }(s)}{I_{1}^{5+\delta }(s)}\text{.}
\end{equation*}%
Hence,%
\begin{equation}
\begin{aligned}
F(N_{1},N_{2},N_{5}) &=N_{1}\log \left( \frac{N_{1}^{5+\delta }}{m^{\delta
/2}N_{2}^{4+\delta }}\dfrac{2^{\delta /2}I_{2}^{4+\delta }(s)}{%
I_{1}^{5+\delta }(s)\Gamma \left( \delta /2\right) \pi e^{1+\delta /2}}%
e^{-\theta (s)}\right) \text{,}  \\
\theta (s) &=\frac{sI_{0}(s)+e^{-s^{2}}}{2sI_{0}(s)+e^{-s^{2}}}\text{.}
\end{aligned}
\label{AR}
\end{equation}

\section{Entropy production estimate \label{EPE}}
Denote by the functional%
\begin{equation*}
\mathfrak{D}(f)=-\int\limits_{\mathbb{R}^{3}\times \mathbb{R}_{+}^2}Q_{\delta
}(f,f)\log \left( I^{1-\delta /2}f\right) \,d\boldsymbol{\xi}\,%
dI\,dx\geq 0
\end{equation*}%
the total entropy production. Define, in view of the inequality~\eqref{PIE4}, 
\begin{equation*}
\begin{aligned}
&\Lambda (m^{-1}p_{\infty },m^{-1}k_{B}T_{\infty },\mathcal{M}_{\infty};M_{0+}(\boldsymbol{\xi },I)) \\
&=\frac{p_{0}}{\sqrt{2\pi mk_{B}T_{0}}}\log \left( \frac{m^{3/2}p_{0}\sqrt{e}}{\left( 2\pi \right) ^{3/2}\Gamma \left( \delta /2\right)
\left(ek_{B}T_{0}\right) ^{\left( 5+\delta \right) /2}}\right) -F(N_{1}^{-},N_{2}^{-},N_{5}^{-}) 
\\ & \qquad 
+p_{\infty }\sqrt{\frac{\gamma }{mk_{B}T_{\infty }}}\mathcal{M}_{\infty
}\log \left( \frac{\left( 2\pi \right) ^{3/2}\Gamma \left( \delta /2\right)
\left( ek_{B}T_{\infty }\right) ^{\left( 5+\delta \right) /2}}{%
m^{3/2}p_{\infty }e}\right) \geq 0\text{.}
\end{aligned}
\end{equation*}%
Then we have an upper bound of the total entropy production%
\begin{equation*}
0\leq \mathfrak{D}(f)\leq \Lambda (m^{-1}p_{\infty },m^{-1}k_{B}T_{\infty },%
\mathcal{M}_{\infty };M_{0+}(\boldsymbol{\xi },I))
\end{equation*}%
Straightforward calculations give that%
\begin{multline*}
\Lambda (\frac{p_{\infty }}{m},\frac{k_{B}T_{\infty }}{m},\mathcal{M}%
_{\infty };p_{0}\frac{m^{\left( 3+\delta \right) /2}}{\left(
k_{B}T_{0}\right) ^{\left( 5+\delta \right) /2}}M_{0+}\left( \left( \frac{%
k_{B}T_{0}}{m}\right) ^{1/2}\boldsymbol{\xi },\frac{k_{B}T_{0}}{m}I\right) \\
=\frac{p_{0}}{\sqrt{mk_{B}T_{0}}}\left. \Lambda \left( \frac{p_{\infty }}{%
p_{0}},\frac{T_{\infty }}{T_{0}},\mathcal{M}_{\infty };M_{0+}(\boldsymbol{%
\xi },I)\right) \right\vert _{m=1}
\end{multline*}%
for any positive constants $p_{0}$ and $T_{0}$. We consider the normalized
Maxwellian%
\begin{equation*}
\widehat{M}_{0}(\boldsymbol{\xi },I)=\frac{I^{\delta /2-1}}{\left( 2\pi
\right) ^{3/2}\Gamma \left( \delta /2\right) }\exp \left( -\frac{\left\vert 
\boldsymbol{\xi }\right\vert }{2}-I\right) \text{,}\hspace*{3mm}\xi _{1}>0%
\text{,}
\end{equation*}%
at the interface. Then, we have (dropping the index $\infty $)%
\begin{multline}
\Lambda (p,T,\mathcal{M};\widehat{M}_{0+}(\boldsymbol{\xi },I))=p\sqrt{\frac{%
\gamma }{T}}\mathcal{M}\log \left( \frac{\left( 2\pi \right) ^{3/2}\Gamma
\left( \delta /2\right) e^{\left( 3+\delta \right) /2}T^{\left( 5+\delta
\right) /2}}{p}\right) \\
-\left. F(N_{1}^{-},N_{2}^{-},N_{5}^{-})\right\vert _{m=1}-\frac{1}{\sqrt{%
2\pi }}\left( \log \left( \left( 2\pi \right) ^{3/2}\Gamma \left( \delta
/2\right) e^{2+\delta /2}\right) \right) \geq 0\text{,}  \label{NE}
\end{multline}%
where%
\begin{equation}
\begin{aligned}
N_{1}^{-} &=N_{1}^{+}-\mathcal{L}_{1}=\frac{1}{\sqrt{2\pi }}-p\sqrt{\frac{%
\gamma }{T}}\mathcal{M}
\text{,} \\
N_{2}^{-} &=\mathcal{L}_{2}-N_{2}^{+}=p\left( 1+\gamma \mathcal{M}^{2}\right) -\frac{1}{2}
\text{,} \\
N_{5}^{-} &=N_{5}^{+}-\mathcal{L}_{5}=\frac{4+\delta }{\sqrt{2\pi }}-p\sqrt{%
\gamma T}\mathcal{M}\left( 5+\delta +\gamma \mathcal{M}^{2}\right)\text{,}
\end{aligned}
\label{NM}
\end{equation}%
$F(N_{1},N_{2},N_{5})$ is given by equation $\left( \ref{AR}\right) $, and $s$ is implicitly given by (cf. equality~\eqref{AE})
\begin{equation*}
\frac{I_{1}(s)}{I_{2}(s)}\left( s+\frac{I_{1}(s)}{2I_{2}(s)}\left( 4+\delta
\right) \right) =\frac{I_{1}(s)\left[ 2I_{3}(s)+I_{1}(s)\left( 2+\delta
\right) \right] }{2I_{2}^{2}(s)}=\Upsilon \text{,}
\end{equation*}%
with%
\begin{equation*}
\Upsilon =\Upsilon (p,T,\mathcal{M})=\frac{2\left( 1-p\sqrt{2\pi \gamma /T}%
\mathcal{M}\right) \left( 4+\delta -p\sqrt{2\pi \gamma T}\mathcal{M}\left(
5+\delta +\gamma \mathcal{M}^{2}\right) \right) }{\pi \left( 2p\left(
1+\gamma \mathcal{M}^{2}\right) -1\right) ^{2}}\text{.}
\end{equation*}%
Then, since $\dfrac{I_{1}(s)}{I_{2}(s)}$ is nonnegative, 
\begin{equation*}
\begin{aligned}
\frac{I_{1}(s)}{I_{2}(s)} &=
\sqrt{\frac{s^{2}}{\left( 4+\delta \right) ^{2}} + 
\frac{2\Upsilon }{4+\delta }}-\frac{s}{4+\delta }=\frac{\sqrt{s^{2}+2\left(4+\delta \right) \Upsilon }-s}{4+\delta }
\\ &=
\frac{2\Upsilon }{\sqrt{s^{2}+2\left( 4+\delta \right) \Upsilon }+s}\text{,}
\end{aligned}
\end{equation*}
or, equivalently,%
\begin{equation}
2\frac{I_{2}(s)}{I_{1}(s)}=\frac{\sqrt{s^{2}+2\left( 4+\delta \right)
\Upsilon }+s}{\Upsilon }\text{, }  \label{AR1}
\end{equation}%
while, by the relations~\eqref{AIR},
\begin{equation}
\frac{I_{0}(s)}{I_{1}(s)}=-2s+2\frac{I_{2}(s)}{I_{1}(s)} = 
\frac{\sqrt{s^{2}+2\left( 4+\delta \right) \Upsilon } +
\left( 1-2\Upsilon \right) s}{\Upsilon }\text{.}
\label{AR2}
\end{equation}
Hence, by relations~\eqref{AIR}, we obtain%
\begin{equation}
I_{0}(s)=\frac{e^{-s^{2}}}{2}\frac{\sqrt{s^{2}+2\left( 4+\delta \right)
\Upsilon }+\left( 1-2\Upsilon \right) s}{\left( 1+2s^{2}\right) \Upsilon
-s^{2}-s\sqrt{s^{2}+2\left( 4+\delta \right) \Upsilon }}\text{.}
\label{AR3}
\end{equation}

\section{Numerical results \label{NR}}
The allowed physical domain of positive entropy production in the $\left(p,T,\mathcal{M}\right)$-space is bounded by the surface
\begin{equation*}
    S:\Lambda (p,T,\mathcal{M};\widehat{M}_{0+}(\boldsymbol{\xi },I))=0.
\end{equation*}
Applying expression~\eqref{AR}, $\Lambda $, given by relation~\eqref{NE},~\eqref{NM}, can be recast as
\begin{multline*}
\Lambda (p,T,\mathcal{M};\widehat{M}_{0+}(\boldsymbol{\xi },I))=p\sqrt{\frac{%
\gamma }{T}}\mathcal{M}\log \frac{T^{\left( 5+\delta \right) /2}}{p}-\frac{1%
}{\sqrt{2\pi }}\log \sqrt{e} \\
-\left( \frac{1}{\sqrt{2\pi }}-p\sqrt{\frac{\gamma }{T}}\mathcal{M}\right)
\log \left( \frac{2^{\left( 5+\delta \right) /2}\sqrt{\pi }\left( \dfrac{1}{%
\sqrt{2\pi }}-p\sqrt{\dfrac{\gamma }{T}}\mathcal{M}\right) ^{5+\delta
}\Delta e^{-\widetilde{\theta }(s)}}{\left( 2p\left( 1+\gamma \mathcal{M}%
^{2}\right) -1\right) ^{4+\delta }\Upsilon ^{5+\delta }}\right) \geq 0\text{,}
\end{multline*}%
where, in view of expressions ~\eqref{AT},~\eqref{AR1},~\eqref{AR2},~\eqref{AR3},
\begin{align*}
\widetilde{\theta }(s) &=\theta (s)-\frac{1}{2}-s^{2}=\frac{1}{2}-\frac{%
s^{2}+s\sqrt{s^{2}+2\left( 4+\delta \right) \Upsilon }}{2\Upsilon } \\
\Delta  &=\left( \Upsilon +s^{2}\left( 2\Upsilon -1\right) -s\sqrt{%
s^{2}+2\left( 4+\delta \right) \Upsilon }\right) \left( s+\sqrt{%
s^{2}+2\left( 4+\delta \right) \Upsilon }\right) ^{4+\delta } \\
\Upsilon &= %\Upsilon (p,T,\mathcal{M}) = 
\frac{2\left( 1-p\sqrt{\dfrac{2\pi \gamma }{T}}\mathcal{M}\right) \left( 4+\delta -p\sqrt{2\pi \gamma T}%
\mathcal{M}\left( 5+\delta +\gamma \mathcal{M}^{2}\right) \right) }{\pi
\left( 2p\left( 1+\gamma \mathcal{M}^{2}\right) -1\right) ^{2}}>1\text{,}
\end{align*}%
and $s$ is implicitly given by equation~\eqref{AE}.
%\begin{equation}
%\int\limits_{-s}^{\infty }e^{-z^{2}}\mathbf{\,}dz=\frac{e^{-s^{2}}}{2}\frac{%
%\sqrt{s^{2}+2\left( 4+\delta \right) \Upsilon }+\left( 1-2\Upsilon \right) s%
%}{\Upsilon +s^{2}\left( 2\Upsilon -1\right) -s\sqrt{s^{2}+2\left( 4+\delta
%\right) \Upsilon }}\text{.}  \label{l15}
%\end{equation}%
Here, the expression for $\Lambda $ can be recast as%
\begin{multline*}
\Lambda (p,T,\mathcal{M};\widehat{M}_{0+}(\boldsymbol{\xi },I))=\frac{1}{\sqrt{%
2\pi T}}\sqrt{T}\log \frac{T^{\left( 5+\delta \right) /2}}{p\sqrt{e}} \\
+\frac{1}{\sqrt{2\pi T}}\left( \sqrt{2\pi \gamma }p\mathcal{M-}\sqrt{T}%
\right) \log \left( \frac{2\left( \sqrt{\pi }\left( p\left( 1+\gamma 
\mathcal{M}^{2}\right) -1/2\right) \right) ^{6+\delta }T^{\left( 5+\delta
\right) /2}\Delta e^{-\widetilde{\theta }(s)}}{p\left( 4+\delta -p\sqrt{2\pi
\gamma T}\mathcal{M}\left( 5+\delta +\gamma \mathcal{M}^{2}\right) \right)
^{5+\delta }}\right) \geq 0\text{.}
\end{multline*}%

All domains presented in the following for the monatomic case $\delta =0$ were already
presented in \cite{BGH-01}, but are reproduced here as a comparison to give an
indication of the influence of polyatomicity. 

\subsection{Evaporation}
\begin{figure}
    \centering
    \includegraphics[width=\textwidth]{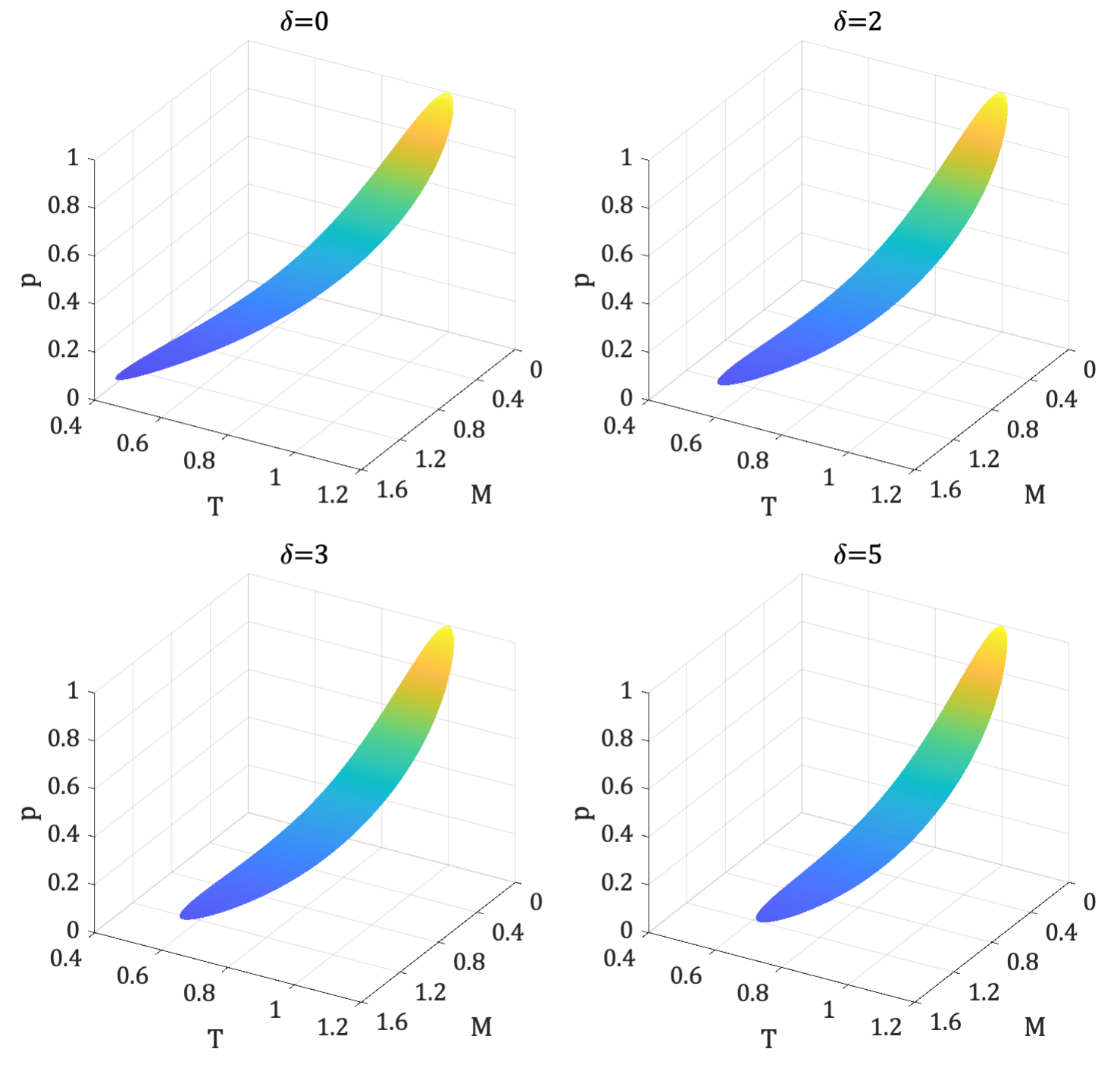}
    \caption{The pipe-like boundary surface $S$ in case of evaporation for $\delta=0,2,3,$ and $5$, respectively.
    The coloring depends on the $p$-value and is scaled individually for each figure.
    Blue represents the minimum value and yellow the maximum.}
    \label{fig:evaporationS}
\end{figure}    
In this section, we consider evaporation, that is, when $\mathcal{M}\geq 0$.
Figure~\ref{fig:evaporationS} shows the boundary surface $S$ for different numbers of internal degrees of freedom $\delta$ also taking into account the additional conditions~\eqref{NC1} and~\eqref{NCP1}. 
The color of the surface encodes the $p$-value and is scaled individually for each $\delta$.
Blue represents the minimum and yellow the maximum value.
For each $\delta$, the surface $S$ is a thin pipe-like structure, and the physical domain of positive entropy production comprises the area inside this pipe.

To further study the physical domain of positive entropy production, we next limit our attention to the case of subsonic evaporation, with the parameter $0\le \mathcal{M}\leq 1$ fixed, it is clear that there exist $p=p_{\#}\left(\mathcal{M}\right) $ and $T=T_{\#}\left( \mathcal{M}\right) $, such that
\begin{equation*}
  \Lambda\left(p_{\#}\left( \mathcal{M}\right) ,T_{\#}\left( \mathcal{M}\right),
  \mathcal{M};\widehat{M}_{0+}(\boldsymbol{\xi },I) \right) =
  \max_{p,T}\Lambda\left(p,T,\mathcal{M};\widehat{M}_{0+}(\boldsymbol{\xi },I) \right)\text{.}
\end{equation*}
We numerically construct curves%
\begin{equation*}
    p=p_{\#}\left( \mathcal{M}\right) \text{, }T=T_{\#}\left( \mathcal{M}\right) 
\end{equation*}%
of maximal entropy production in the $\left( p,T,\mathcal{M}\right) $-space.
\begin{figure}
    \centering
    \includegraphics[width=\textwidth]{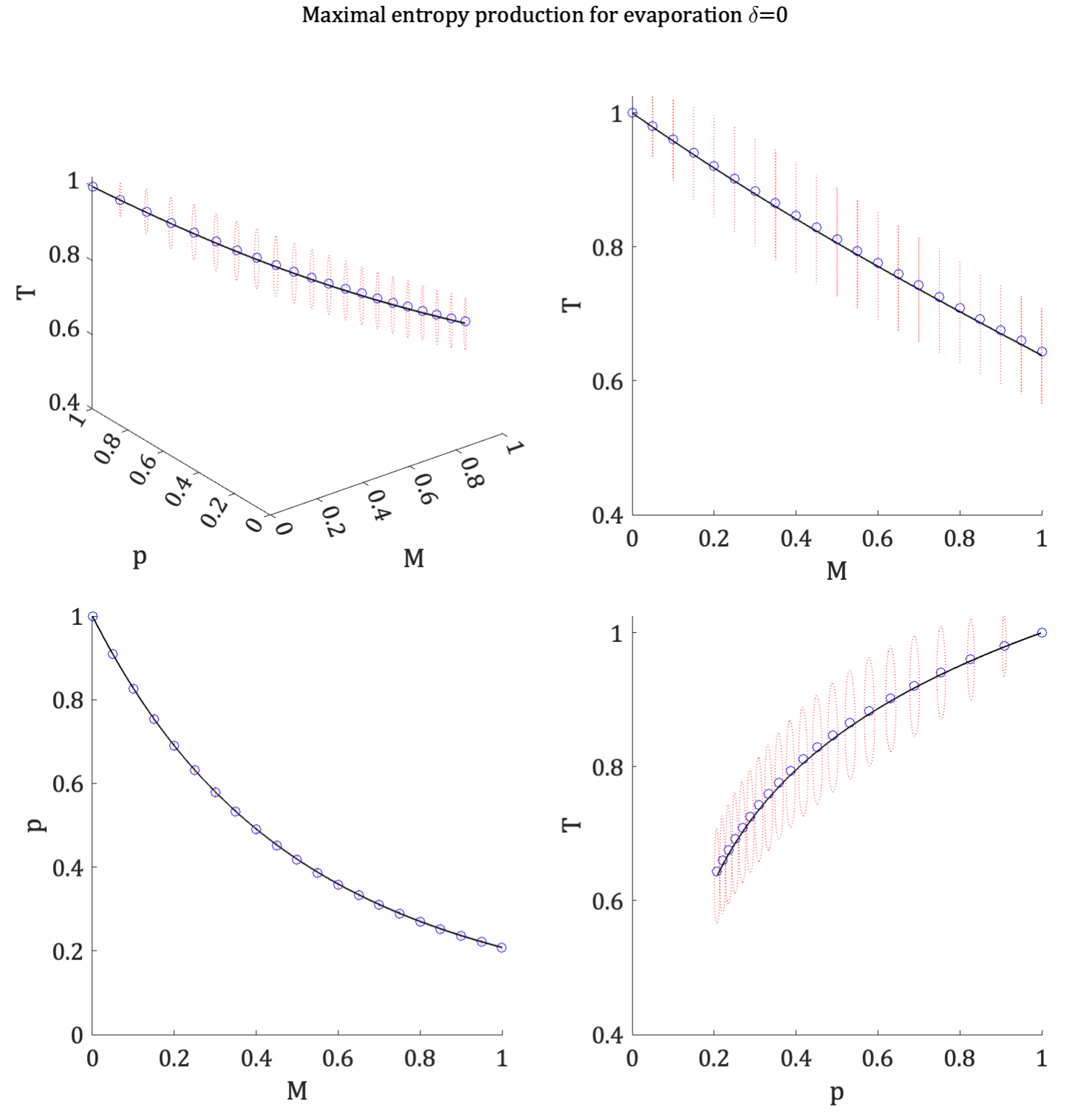}
    \caption{Maximal entropy production curve (solid lines), evaporation curve tabulated by Aoki and Sone~\cite{AS-91} (circles), and cross sections of the boundary surface $S$ (dashed lines) for subsonic evaporation of a monatomic gas.}
    \label{fig:maxEntropyEvaporation0}
\end{figure}    
Figure~\ref{fig:maxEntropyEvaporation0} shows the curves of maximal entropy production in the $\left( p,T,\mathcal{M}\right) $-space for subsonic evaporation of a monatomic gas. 
The solid lines are the curves obtained and the circles indicate the famous evaporation curve tabulated by Aoki and Sone~\cite{AS-91} based on a vast number of numerical solutions of the BGK-equation. The dashed closed lines illustrate the cross sections of the surface $S$ in Figure~\ref{fig:evaporationS} computed for the same values of $\mathcal{M}$ as the circles.
The top left panel shows these curves and markers in the $\left( p,T,\mathcal{M}\right) $-space.
For increased visibility, the projection of the curves on the $\mathcal{M}$--$T$-plane, the $\mathcal{M}$--$p$-plane, and the $p$--$T$-plane  are
presented in the top right, bottom left, and bottom right panels, respectively.
We note that although the pipe-like structure of $S$ may not be immediately clear from the plots in Figure~\ref{fig:evaporationS}, it becomes apparent when studying its cross sections from the different viewpoints illustrated in Figure~\ref{fig:maxEntropyEvaporation0}.
For each fixed $\mathcal{M}$, the pipe is very narrow in the $p$-direction and significantly more extended in the $T$-direction.

\begin{figure}
\centering
    \includegraphics[width=\textwidth]{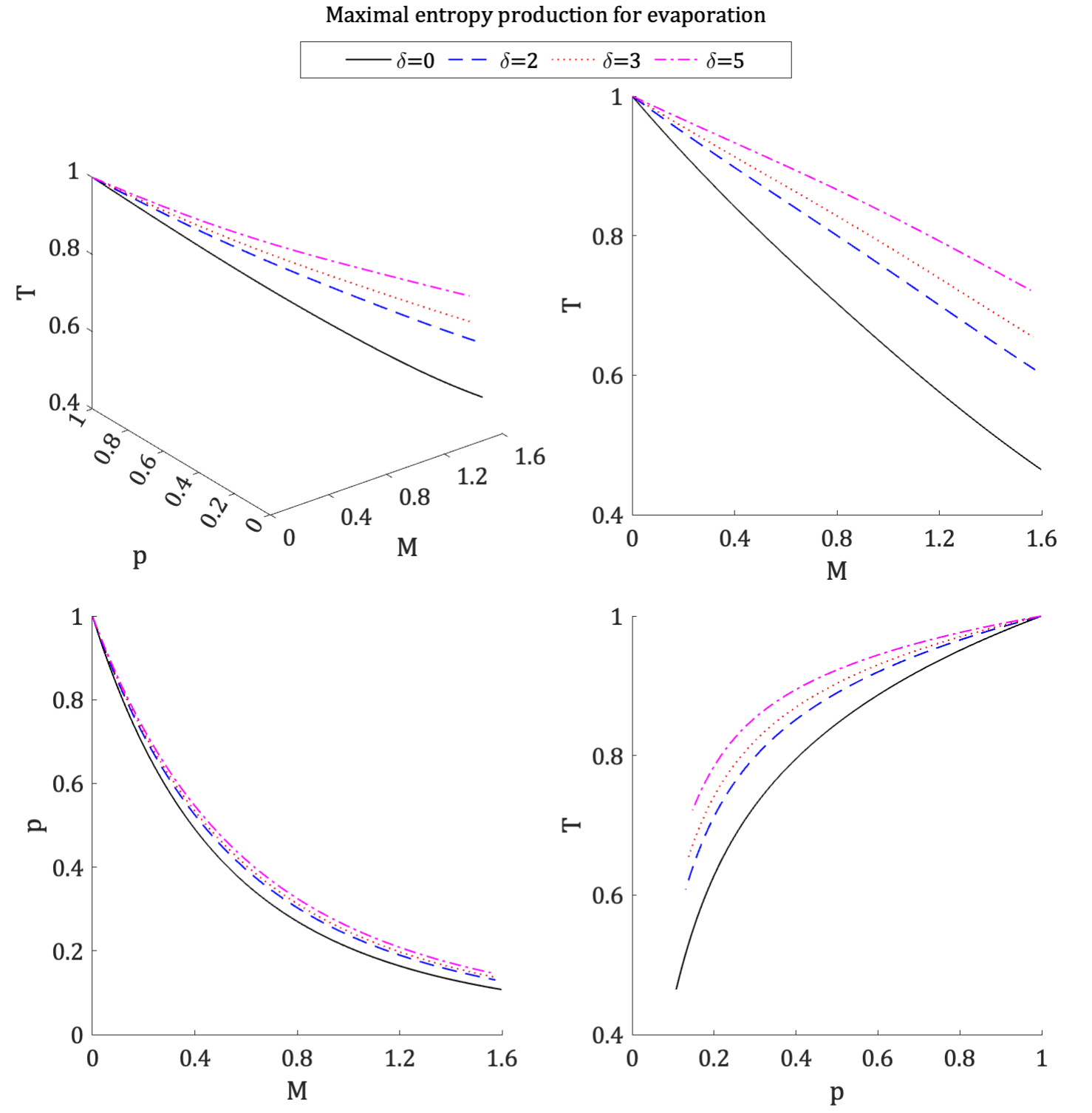}
    \caption{Curves of maximal entropy production in the $\left( p,T,\mathcal{M}\right) $-space in the case of evaporation for $\delta=0,2,3,$ and $5$.}
    \label{fig:maxEntropyEvaporation}
\end{figure}
As our next objective, we remove the upper bound on $\mathcal{M}$ and search for maximal entropy production curves.
Figure~\ref{fig:maxEntropyEvaporation} presents numerically computed curves of maximal entropy production for different numbers of internal degrees of freedom $\delta$.
The different lines represent different numbers of degrees of freedom; the solid, dashed, dotted, and dash--dotted lines correspond to $\delta=0,2,3,$ and $5$, respectively. 
First, we note that the relations above allow positive entropy production for $\mathcal{M}$ up to about 1.6---the maximal value slightly decreasing with the number of internal degrees of freedom.
Moreover, for each fixed $\mathcal{M}$, the pressure $p_{\#}$ and temperature $T_{\#}$ for maximal entropy production increase as the number of degrees of freedom $\delta$ increases.

\subsection{Condensation}
\begin{figure}
    \centering
    \includegraphics[width=\textwidth]{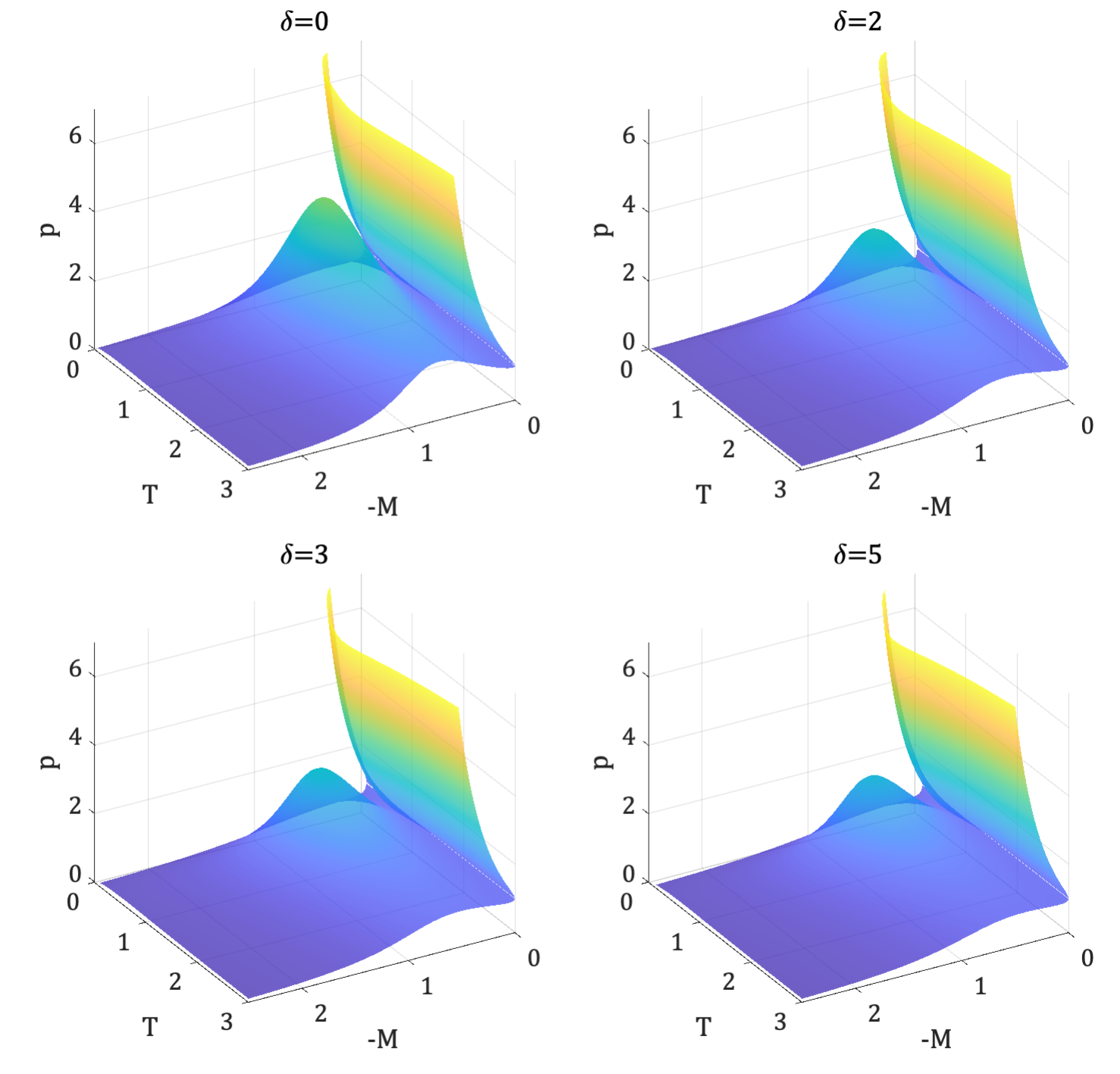}
    \caption{The boundary surface $S$ in case of condensation for $\delta=0,2,3,$ and $5$, respectively.
    %All graphs are cropped so that $p$-values above 7 are not displayed. 
    The coloring depends on the $p$-value, with the same scaling for all the values of $\delta$; blue represents $p=0$ and yellow $p=7$.}
    \label{fig:condensationS}
\end{figure}    
Next, we consider condensation, that is, when $\mathcal{M}<0$.
Figure~\ref{fig:condensationS} illustrates the boundary surface $S$ for different numbers of internal degrees of freedom $\delta$ taking into account the additional conditions~\eqref{NC1} and~\eqref{NCN1}.
All graphs are cropped so that values of $p$ above 7 are not displayed. 
The color depends on the value of $p$, with the same scaling for all the values of $\delta$; blue represents $p=0$ and yellow $p=7$.
For all degrees of freedom studied, the corresponding surfaces share a few common characteristics.
First, for small values of $|\mathcal{M}|$ and any value of $T$, the range of $p$ for the physical domain of positive entropy production has both a lower and an upper bound.
Moreover, in the limit $\mathcal{M}\to 0^-$, both the upper and lower bounds tend to $p=1$ for any $T>0$.
Looking further at the lower bound, it can be seen that, for small $T>0$, the lower bound has a local maximum in $p$ around $\mathcal{M}=-1$.
The value of $p$ for this local maximum decreases as the number of internal degrees of freedom $\delta$ increases.

\begin{figure}
    \centering
    \includegraphics[width=\textwidth]{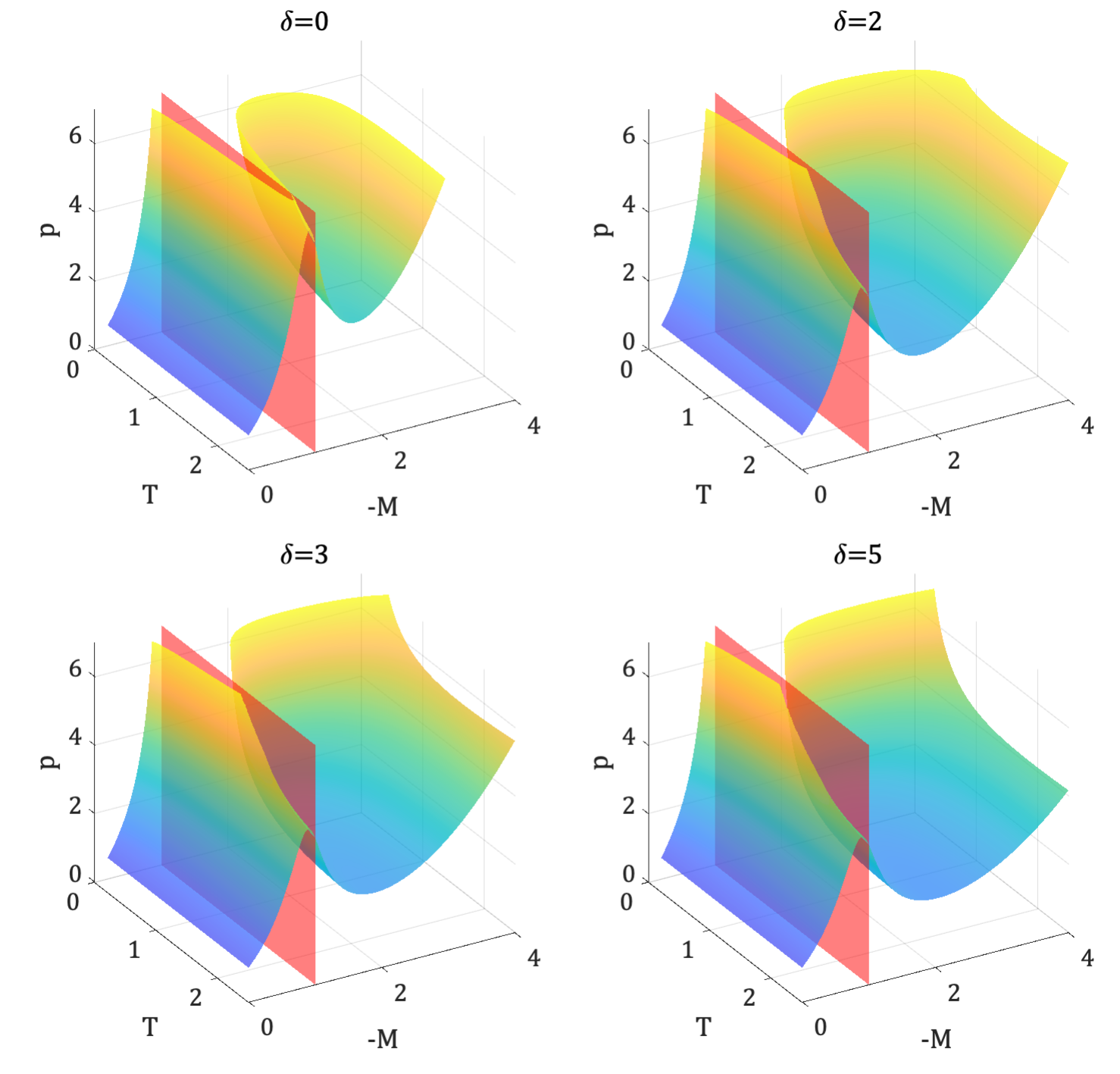}
    \caption{Maximal entropy production surfaces in the $\left( p,T,\mathcal{M}\right)$-space in case of condensation for $\delta=0,2,3,$ and $5$, respectively; illustrated together with the plane $\mathcal{M}=-1$.
    The coloring of the surfaces depends on the $p$-value, with the same scaling for all the values of $\delta$; blue represents $p=0$ and yellow $p=7$.}
    \label{fig:maxEntropyCondensation}
\end{figure}
To further study the physical domain of positive entropy production in the case of condensation, we fix the two parameters $T>0$ and $\mathcal{M}<0$ and note that there exists $p=p_{\ast }\left( T,\mathcal{M} \right) $, such that 
\begin{equation*}
\Lambda (p_{\ast }\left( T,\mathcal{M}\right) ,T,\mathcal{M};\widehat{M}%
_{0+}(\boldsymbol{\xi },I))=\max_{p}\Lambda (p,T,\mathcal{M};\widehat{M}%
_{0+}(\boldsymbol{\xi },I))\text{.}
\end{equation*}%
Thus, for $\mathcal{M}<0$ and $T>0$, we construct the surfaces 
\begin{equation*}
p=p_{\ast }\left( T,\mathcal{M}\right) 
\end{equation*}%
of maximal entropy production in the $\left( p,T,\mathcal{M}\right) $-space.
Figure~\ref{fig:maxEntropyCondensation} shows these surfaces in the $\left( p,T,\mathcal{M}\right)$-space for different numbers of internal degrees of freedom together with the plane $\mathcal{M}=-1$.
For all degrees of freedom studied, the surfaces share similar characteristics: with a clear increase in $p_\ast$ as $|\mathcal{M}|$ increases in the range $\mathcal{M}<0$---starting from the line $p_\ast=1$ as $\mathcal{M}$ tends to $0^-$---up to about (somewhat less than) $\mathcal{M}=-1$.
Looking at the intersection of the surface of maximal entropy production with the plane $\mathcal{M}=-1$, we see that for a larger number of internal degrees of freedom $\delta$, the decrease of $p_\ast$ is faster as $T$ increases.
For $\mathcal{M}<-1$, the surface of maximal entropy production has a well-shaped structure, the width of this well in the $\mathcal{M}$-direction increases as $\delta$ increases.

\begin{figure}
    \centering
    \includegraphics[width=\textwidth]{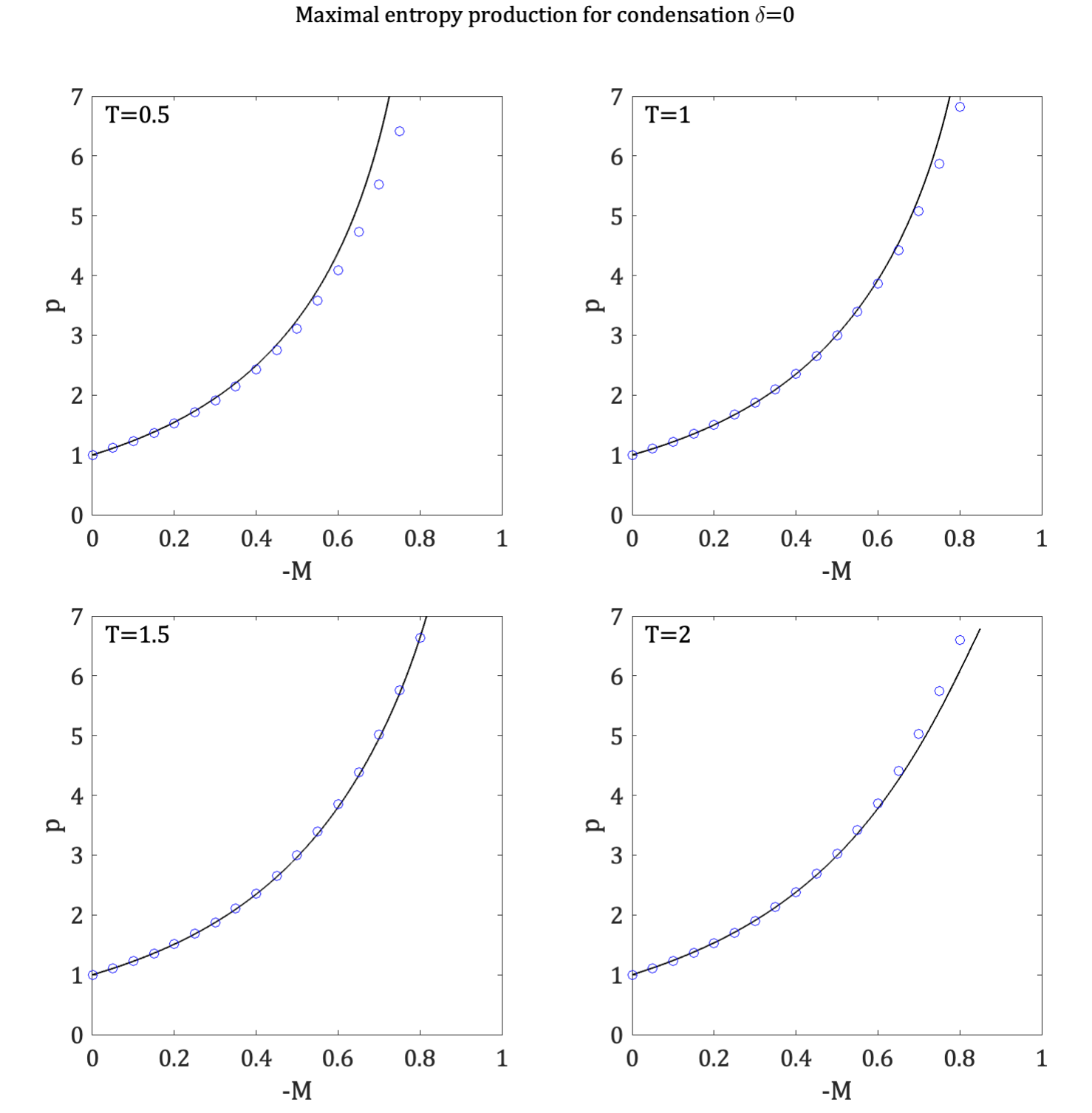}
    \caption{Maximal entropy production curve (solid lines) in the $\mathcal{M}$--$p$-plane for $T=0.5,1,1.5,$ and $2$ compared to the data tabulated by Aoki and Sone~\cite{AS-91} (circles) for subsonic condensation of a monatomic gas.}
    \label{fig:maxEntropyCondensation0}
\end{figure}
Figure~\ref{fig:maxEntropyCondensation0} shows the intersection curves of the surface of maximal entropy production and the planes $T=0.5,T=1,T=1.5,$ and $T=2$ for subsonic condensation of a monatomic gas. 
The solid lines represent our computed curves, while the circles indicate the corresponding intersection curves from the well-known condensation surface tabulated by Aoki and Sone~\cite{AS-91} based on a large number of numerical solutions of the BGK equation. 
We observe that all the intersection curves agree well with the tabulated values for small values of $|\mathcal{M}|$, while the agreement gradually decreases as $|\mathcal{M}|$ increases.
The maximal entropy curve increases faster than the tabulated curve for $T=0.5$ and $T=1$, but with a smaller discrepancy for $T=1$, increases slower than the tabulated curve in some part and faster in some other part of the interval for $T=1.5$, and then slower than the tabulated curve for $T=2$. 
Overall, the agreement is particularly strong for $T=1.5$, while for $T=1$ the agreement is better up to a certain value of $|\mathcal{M}|$. 

\begin{figure}
    \centering
    \includegraphics[width=\textwidth]{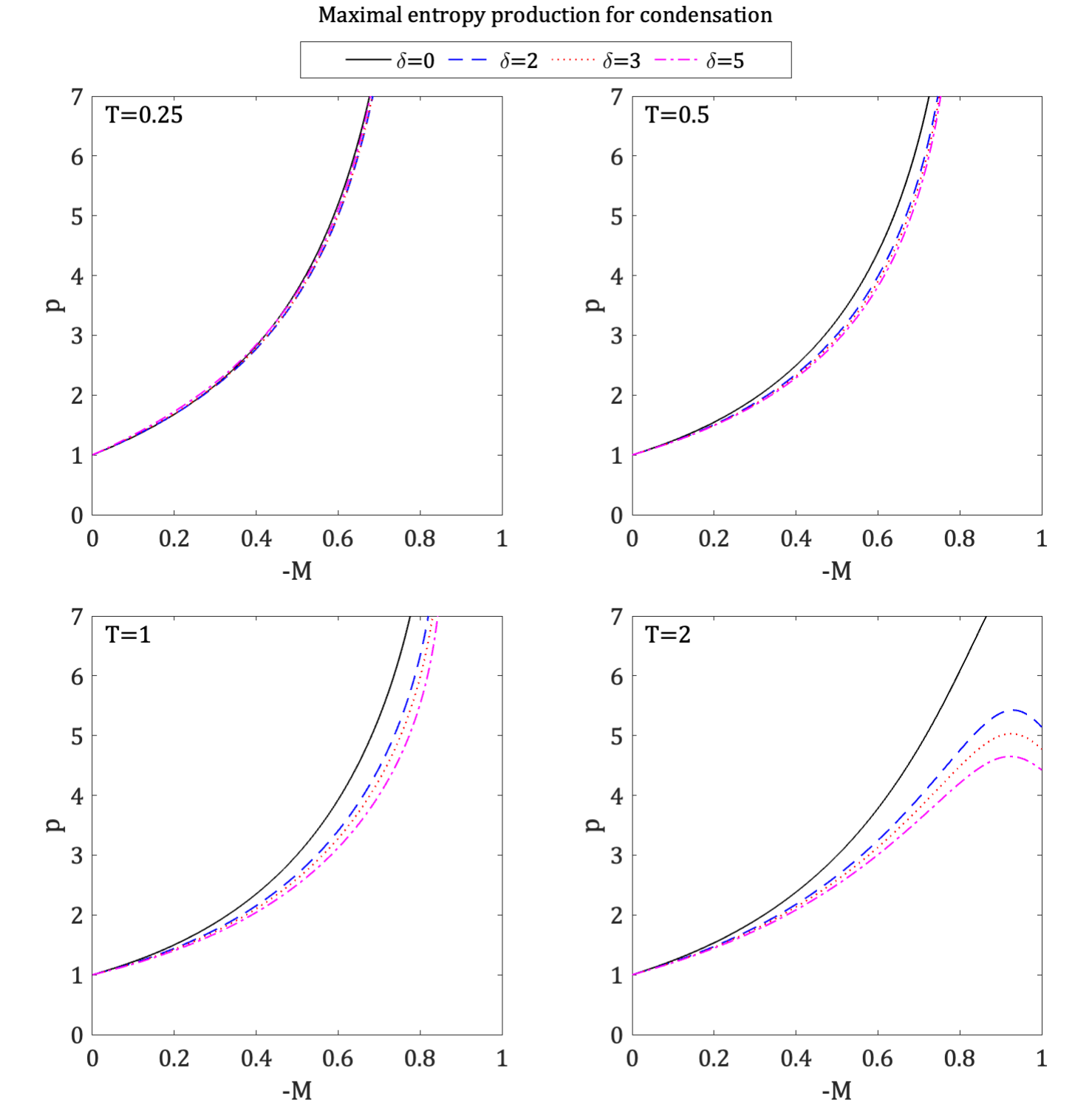}
    \caption{Maximal entropy production curve in the $\mathcal{M}$--$p$-plane for $\delta=0,2,3,$ and $5$, respectively for $T=0.25,0.5,1.0,$ and $2$ for subsonic condensation.}
    \label{fig:maxEntropyCondensationSlices}
\end{figure}   
Figure~\ref{fig:maxEntropyCondensationSlices} shows the intersection curves of the surface of maximal entropy production with the planes $T=0.25,T=0.5,T=1,$ and $T=2$ in the case of condensation for different numbers of degrees of freedom.
The different lines represent different numbers of internal degrees of freedom; the solid, dashed, dotted, and dash--dotted lines correspond to $\delta=0,2,3,$ and $5$, respectively. 
First, we note that all the intersection curves agree well for all the numbers of internal degrees of freedom for small values of $|\mathcal{M}|$, while the differences become more pronounced as $|\mathcal{M}|$ increases. The agreement of the maximal entropy curves is fairly good for $T = 0.25$.
As $T$ increases, the differences between the curves increase, with lower values of $p^\ast$ observed for larger values of $\delta$. 

\section*{Acknowledgements} 
N.B. acknowledges travel grants from the French Institute in Sweden through the FRÖ program and SVeFUM. S.B. acknowledges a travel grant from the French Institute in Sweden through the TOR program. N.B. gratefully thanks Kazuo Aoki for valuable discussions and comments on the manuscript.


\begin{thebibliography}{10}

\bibitem{ACG-24}
R.~Alonso, M.~Colic, and I.~Gamba, {\em The {C}auchy problem for {B}oltzmann
  bi-linear systems: The mixing of monatomic and polyatomic gases}, J. Stat.
  Phys., 191 (2024), pp.~9:1--50.

\bibitem{Anderson-03}
J.~D. Anderson, {\em Modem compressible flow: with historical perspective}, 
  McGraw Hill, 3rd~ed., 2003.

\bibitem{Anderson-06}
J.~D. Anderson, {\em Hypersonic and high-temperature gas dynamics}, American Institute of Aeronautics and Astronautics, Inc., 2nd~ed., 2006.

\bibitem{ANSS-91}
K.~Aoki, K.~Nishino, Y.~Sone, and H.~Sugimoto, {\em Numerical analysis of
  steady flows of a gas condensing on or evaporating from its plane condensed
  phase on the basis of kinetic theory: {E}ffect of gas motion along the
  condensed phase}, Phys. Fluids A, 32 (1991), pp.~2260--2275.

\bibitem{AS-91}
K.~Aoki and Y.~Sone, {\em Gas flows around the condensed phase with strong
  evaporation or condensation: fluid dynamic equation and its boundary
  condition on the interface and their application}, in Advances in Kinetic
  Theory and Continuum Mechanics, R.~Gatignol and Soubbaramayer, eds.,
  Springer-Verlag, 1991, pp.~43--54.

\bibitem{ASY-90}
K.~Aoki, Y.~Sone, and T.~Yamada, {\em Numerical analysis of gas flows
  condensing on its plane condensed phase on the basis of kinetic theory},
  Phys. Fluids A, 2 (1990), pp.~1867--1878.

\bibitem{AMR-24}
T.~Arima, A.~Mentrelli, and T.~Ruggeri, {\em A novel {ES-BGK} model for
  non-polytropic gases with internal state density independent of temperature},
  J. Stat. Phys., 191 (2024), pp.~95:1--34.

\bibitem{Atkins-10}
P.~Atkins and J.~de~Paula, {\em Physical Chemistry}, W. H. Freeman and Company,
  9th~ed., 2010.

\bibitem{BBBD-18}
C.~Baranger, M.~Bisi, S.~Brull, and L.~Desvillettes, {\em On the
  {C}hapman-{E}nskog asymptotics for a mixture of monatomic and polyatomic
  rarefied gases}, Kinet. Relat. Models, 11 (2018), pp.~821--858.

\bibitem{BGS-06}
C.~Bardos, F.~Golse, and Y.~Sone, {\em Half-space problems for the {B}oltzmann
  equation: A survey}, J. Stat. Phys., 124 (2006), pp.~275--300.

\bibitem{Be-23d}
N.~Bernhoff, {\em Linear half-space problems in kinetic theory: Abstract
  formulation and regime transitions}, Int. J. Math., 34 (2023),
  pp.~2350091:1--41.

\bibitem{Be-23a}
N.~Bernhoff, {\em Linearized
  {B}oltzmann collision operator: {I.} {P}olyatomic molecules modeled by a
  discrete internal energy variable and multicomponent mixtures}, Acta Appl.
  Math., 183 (2023), pp.~3:1--45.

\bibitem{Be-23b}
N.~Bernhoff, {\em Linearized
  {B}oltzmann collision operator: {II.} {P}olyatomic molecules modeled by a
  continuous internal energy variable}, Kinet. Relat. Models, 16 (2023),
  pp.~828--849.

\bibitem{Be-24a}
N.~Bernhoff, {\em Compactness property
  of the linearized {B}oltzmann collision operator for a mixture of monatomic
  and polyatomic species}, J. Stat. Phys., 191 (2024), pp.~32:1--35.

\bibitem{Be-24b}
N.~Bernhoff, {\em Compactness property
  of the linearized {B}oltzmann collision operator for a multicomponent
  polyatomic gas}, J. Math. Anal. Appl., 537 (2024), pp.~128265:1--31.

\bibitem{Be-24c}
N.~Bernhoff,  {\em Linearized
  {B}oltzmann collision operator for a mixture of monatomic and polyatomic
  chemically reacting species}, J. Math. Chem., 62 (2024), pp.~1935--1964.

\bibitem{BBCG-25}
N.~Bernhoff, L.~Boudin, M.~Colic, and B.~Grec, {\em Compactness of linearized
  {B}oltzmann operators for polyatomic gases}, in Mathematical Models for
  Interacting Dynamics on Networks (Mat-Dyn-Net), M.~Colic, J.~Giesselmann,
  J.~Glück, M.~K. Fijavz, A.~Mauroy, and D.~Mugnolo, eds., Springer, 2025,
  p.~27pp.
\newblock accepted; arXiv: 2407.11452.

\bibitem{BG-21}
N.~Bernhoff and F.~Golse, {\em On the boundary layer equations with phase
  transition in the kinetic theory of gases}, Arch. Ration. Mech. Anal., 240
  (2021), pp.~51--98.

\bibitem{BBG-24}
M.~Bisi, T.~Borsoni, and M.~Groppi, {\em An internal state
  kinetic model for chemically reacting mixtures of monatomic and polyatomic
  gases}, Kinet. Relat. Models, 17 (2024), pp.~276--311.

\bibitem{BGH-01}
A.~Bobylev, R.~Grzhibovskis, and A.~Heintz, {\em Entropy inequalities for
  evaporation/condensation problem in rarefied gas dynamics}, J. Stat. Phys.,
  102 (2001), pp.~1151--1176.

\bibitem{BL-75}
C.~Borgnakke and P.~Larsen, {\em Statistical collision model for
  {M}onte-{C}arlo simulation of polyatomic mixtures}, J. Comput. Phys., 18
  (1975), pp.~405--420.

\bibitem{BBG-21}
T.~Borsoni, M.~Bisi, and M.~Groppi, {\em A general framework for the kinetic
  modelling of polyatomic gases}, Commun. Math. Phys., 393 (2021),
  pp.~215--266.

\bibitem{BDLP-94}
J.-F. Bourgat, L.~Desvillettes, P.~{Le Tallec}, and B.~Perthame, {\em
  Microreversible collisions for polyatomic gases and {B}oltzmann's theorem},
  Eur. J. Mech. B/Fluids, 13 (1994), pp.~237--254.

\bibitem{BST-24}
S.~Brull, M.~Shahine, and P.~Thieullen, {\em Fredholm property of the
  linearized {B}oltzmann operator for a polyatomic single gas model}, Kinet.
  Relat. Models, 17 (2024), pp.~234--252.

\bibitem{De-97}
L.~Desvillettes, {\em Sur un mod\`{e}le de type {B}orgnakke-{L}arsen conduisant
  \`{a} des lois d'\`{e}nergie non-lin\`{e}aires en temp\`{e}rature pour les
  gas parfaits polyatomiques}, Ann. Fac. Sci. Toulouse Math., 6 (1997),
  pp.~257--262.

\bibitem{DMS-05}
L.~Desvillettes, R.~Monaco, and F.~Salvarani, {\em A kinetic model allowing to
  obtain the energy law of polytropic gases in the presence of chemical
  reactions}, Eur. J. Mech. B/Fluids, 24 (2005), pp.~219--236.

\bibitem{DPT-21}
V.~Djordjic, M.~Pavic-Colic, and M.~Torrilhon, {\em {B}oltzmann collision
  operator for polyatomic gases in agreement with experimental data and {DSMC}
  method}, Phys. Rev. E, 104 (2021), pp.~025310:1--7.

\bibitem{DL-23}
R.~Duan and Z.~Li, {\em Global bounded solutions to the {B}oltzmann equation
  for a polyatomic gas}, Int. J. Math., 34 (2023), pp.~2350036:1--43.

\bibitem{ErnGio-94}
A.~Ern and V.~Giovangigli, {\em Multicomponent Transport Algorithms},
  Springer-Verlag, 1994.

\bibitem{GP-23}
I.~Gamba and M.~Pavic-Colic, {\em On the {C}auchy problem for {B}oltzmann
  equation modelling polyatomic gas}, J. Math. Phys., 64 (2023),
  pp.~013303:1--51.

\bibitem{Go-05}
F.~Golse, {\em The {B}oltzmann equation and its hydrodynamic limits}, in
  Handbook of differential equations: Evolutionary equations, C.~Dafermos and
  E.~Feireisl, eds., vol.~2, Elsevier, 2005, pp.~159--301.

\bibitem{Go-08}
F.~Golse, {\em Analysis of the
  boundary layer equation in the kinetic theory of gases}, Bull. Inst. Math.
  Acad. Sin., 3 (2008), pp.~211--242.

\bibitem{GS-99}
M.~Groppi and G.~Spiga, {\em Kinetic approach to chemical reactions and
  inelastic transitions in a rarefied gas}, J. Math. Chem., 26 (1999),
  pp.~197--219.

\bibitem{Sone-02}
Y.~Sone, {\em Kinetic Theory and Fluid Dynamics}, Birkhauser, 2002.

\bibitem{Sone-07}
Y.~Sone, {\em Molecular Gas Dynamics}, Birkhauser, 2007.

\bibitem{SAY-86}
Y.~Sone, K.~Aoki, and I.~Yamashita, {\em A study of unsteady strong
  condensation on a plane condensed phase with special interest in formation of
  steady profile}, in Rarefied gas dynamics, V.~Boffi and C.~Cercignani, eds.,
  vol.~II, B. G. Teubner, 1986, pp.~323--333.

\bibitem{SS-90}
Y.~Sone and H.~Sugimoto, {\em Strong evaporation from a plane condensed phase},
  in Waves in Liquid-Vapor Systems, G.~E.~A. Meier and P.~A. Thompson, eds.,
  Springer-Verlag, 1990, pp.~293--304.

\bibitem{STG-01}
Y.~Sone, S.~Takata, and F.~Golse, {\em Notes on the boundary conditions for
  fluid-dynamic equations on the interface of a gas and its condensed space},
  Phys. Fluids, 13 (2001), pp.~2985--2998.

\bibitem{TG-07}
S.~Takata and F.~Golse, {\em Half-space problem of the nonlinear {B}oltzmann
  equation for weak evaporation and condensation of a binary mixture of
  vapors}, Eur. J. Mech. B/Fluids., 26 (2007), pp.~105--131.

\end{thebibliography}
\end{document}